\documentclass[10pt]{amsart}
\usepackage{graphicx, xcolor}
\usepackage{amssymb}
\numberwithin{equation}{section}

\def\N{\mathbb{N}}
\def\R{\mathbb{R}}

\newcommand{\rdbrack}{]\!]}
\newcommand{\ldbrack}{[\![}

\newtheorem{theorem}{Theorem}[section]
\newtheorem{proposition}[theorem]{Proposition}
\newtheorem{lemma}[theorem]{Lemma}
\newtheorem{corollary}[theorem]{Corollary}
\newtheorem{example}[theorem]{Example}

\title{Metastability for nonlinear parabolic\\ equations with application to scalar\\ viscous conservation laws} 

\begin{document}

\maketitle

\begin{center}
CORRADO MASCIA\footnote{Dipartimento di
  Matematica ``G. Castelnuovo'', Sapienza -- Universit\`a di Roma, P.le Aldo
  Moro, 2 - 00185 Roma (ITALY), 
  {\sc and} {Istituto per le Applicazioni del Calcolo, Consiglio Nazionale delle Ricerche
  (associated in the framework of the program ``Intracellular Signalling'')}, \texttt{mascia@mat.uniroma1.it}}, MARTA STRANI\footnote{Dipartimento di
 Matematica ``G. Castelnuovo'', Sapienza -- Universit\`a di Roma, P.le Aldo
 Moro, 2 - 00185 Roma (ITALY), Tel. +390649913406, E-mail address: \texttt{strani@mat.uniroma1.it}. }
\end{center}
\vskip1cm
\begin{abstract}
The aim article is to contribute to the definition of a versatile language for metastability
in the context of partial differential equations of evolutive type.
A general framework suited for parabolic equations in one dimensional bounded domains is proposed,
based on choosing a family of approximate steady states $\{U^{\varepsilon}(\cdot;\xi) \}_{{}_{\xi \in J}}$
and on the spectral properties of the linearized operators at such states.
The slow motion for solutions belonging to a cylindrical neighborhood of the family $\{U^{\varepsilon}\}$
is analyzed by means of a system of an ODE for the parameter $\xi=\xi(t)$, coupled with a PDE
describing the evolution of the perturbation $v:=u-U^\varepsilon(\cdot;\xi)$.

We state and prove a general result concerning the reduced system for the couple $(\xi,v)$,
called {\sf quasi-linearized system},  obtained by disregarding the nonlinear term in $v$,
and we show how such approach suits to the prototypical example of scalar viscous
conservation laws with Dirichlet boundary condition in a bounded one-dimensional interval 
with convex flux.
\end{abstract}

\maketitle



\pagestyle{myheadings}
\thispagestyle{plain}
\markboth{C. Mascia and M. Strani} {Metastability for nonlinear parabolic equations}

\section{Introduction}\label{Sect:Intro}

Metastability is a broad term describing the existence of a very sensitive 
equilibrium, possessing a weak form of stability/instability.
Usually, such behavior is related to the presence of a small first eigenvalue for the linearized operator 
at the given equilibrium state, revealed at dynamical level by the appearance of slowly moving structures. 
Such circumstance comes into view in the analysis of different classes of evolutive PDEs, and it 
has been object of a wide amount of studies, covering many different areas.
Among others, we emphasize  the explorations on the Allen--Cahn equation, started in 
\cite{CarrPego89, FuscHale89}, and the investigations on the Cahn--Hilliard equation, with the 
fundamental contributions  \cite{Pego89, AlikBateFusc91}.
The analysis has been continued by many other scholars by means of a broad spectrum of techniques,
and extended to a number of different models such as the Gierer--Meinhardt and Gray--Scott 
systems (see \cite{SunWardRuss05}), Keller--Segel chemotaxis system (see \cite{DolaSchm05, PotaHill05}),
general gradient flows (see \cite{OttoRezn07}) and many others.
The number of references is so vast that it would be impossible to mention all the contributions 
given in the area.

A pionereeing article in the analysis of slow dynamics for parabolic equations has been authored
by G.~Kreiss and H.-O.~Kreiss \cite{KreiKrei86} and concerns with the scalar viscous conservation law
\begin{equation}\label{burgers}
	\partial_t u + \partial_x f(u)=\varepsilon\,\partial_x^2 u, \qquad u(x,0)=u_0(x)
\end{equation}
with the space variable $x$ belonging to a one-dimensional interval $I=(-\ell,\ell)$, $\ell>0$.
The primary prototype for the flux function $f$ is given by the classical quadratic formula
$f(u)=\frac12\,u^2$, so that partial differential equation in \eqref{burgers} becomes the
so-called  {\it (viscous) Burgers equation}.
The parameter $\varepsilon>0$ is small.
Problem \eqref{burgers} is complemented with Dirichlet boundary conditions
\begin{equation}\label{bdary}
	u(-\ell,t)=u_-\qquad\textrm{and}\qquad u(\ell,t)=u_+
\end{equation}
for given data $u^\pm$ to be discussed in details.

Burgers equation is considered as a (simplified) archetype of more complicate systems 
of partial differential equations arising in different fields of applied mathematics.
Inspired by the equations of fluid-dynamics, the parameter $\varepsilon$ is interpreted as
a {\it viscosity coefficient} and the main problem is to identify and quantify its r\^ole in the emergence
and/or disappearance of structures.

Formally, in the limit $\varepsilon\to 0^+$, the initial value problem \eqref{burgers}
reduces to a first-order quasi-linear equation of hyperbolic type
\begin{equation}\label{unviscburgers}
	\partial_t u + \partial_x f(u)=0, \qquad u(x,0)=u_0(x)
\end{equation}
whose standard setting is given by the {\it entropy formulation}.
Hence solutions may have discontinuities, which propagate with speed $s$ such that 
\begin{equation*}
 s\ldbrack u\rdbrack=\ldbrack f(u)\rdbrack
 \qquad\qquad
 \textrm{({\sf Rankine--Hugoniot relation})}
\end{equation*}
and satisfy appropriate {\sf entropy conditions} (here $\ldbrack \cdot\rdbrack$ denotes the jump).
In addition, the treatment of the boundary conditions \eqref{bdary} is more delicate
than the parabolic case, because of the eventual appearance of boundary layers,
\cite{BardLeRoNede79}.

Concerning the flux function $f$, let us assume that, for some $c_0>0$,
\begin{equation}\label{fluxhyp}
	f''(u)\geq c_0>0,\qquad f'(u_+)<0<f'(u_-),\qquad f(u_+)=f(u_-),
\end{equation}
where $u_\pm$ are the boundary data prescribed in \eqref{bdary}.
The last two assumptions guarantee that a jump with left value $u_-$ and right value $u_+$ 
satisfy the entropy condition and has speed of propagation equal to zero,
as dictated by the Rankine--Hugoniot relation.
Therefore, the one-parameter family of functions $\{U_{{}_{\textrm{hyp}}}(\cdot;\xi)\}$ defined by
\begin{equation*}
	U_{{}_{\textrm{hyp}}}(x;\xi):=
		u_-\chi_{{}_{(-\ell,\xi)}}(x) +u_+\chi_{{}_{(\xi,\ell)}}(x)
\end{equation*}
(where $\chi_{{}_{I}}$ denotes the characteristic function of the set $I$)
is composed by stationary solutions of the equation in \eqref{unviscburgers}
satisfying the boundary conditions \eqref{bdary}.
The dynamics determined by  boundary-initial value problem \eqref{unviscburgers}-\eqref{bdary}
is simple: for any datum $u_0$ with bounded variation, the solution converges
{\it in finite time} to an element of $\{U_{{}_{\textrm{hyp}}}(\cdot;\xi)\}$ (see Section \ref{Sect:Appl}).
Hence, at the level $\varepsilon=0$, there are infinitely many stationary solutions, generating
a ``finite-time'' attracting manifold for the dynamics.
\vskip.15cm

For $\varepsilon>0$, the situation is different. 
Apart from the well-known smoothing effect, the presence of the Laplace operator in \eqref{burgers}
has the effect of a drastic reduction of the number of stationary solutions satisfying \eqref{bdary}:
from infinitely many to a single stationary state (see Section \ref{Sect:Appl}).
Such solution, denoted here by 
$\bar U_{{}_{\textrm{par}}}^\varepsilon=\bar U_{{}_{\textrm{par}}}^\varepsilon(x)$,
converges in the limit $\varepsilon\to 0^+$ to a specific element  $U_{{}_{\textrm{hyp}}}(\cdot;\bar\xi)$ of the family 
$\{U_{{}_{\textrm{hyp}}}(\cdot;\xi)\}$.

The dynamical properties of \eqref{burgers}--\eqref{bdary} for initial data close to the equilibrium
configuration $\bar U_{{}_{\textrm{par}}}^\varepsilon$ can be analyzed linearizing at the state
$\bar U_{{}_{\textrm{par}}}^\varepsilon$
\begin{equation*}
	\partial_t u=\mathcal{L}_\varepsilon\,u
	:=\varepsilon\,\partial_x^2 u+\partial_x \bigl(a(x)\,u\bigr)
	\qquad\textrm{with}\;a(x):=-f'(\bar U_{{}_{\textrm{par}}}^\varepsilon(x)).
\end{equation*}
In \cite{KreiKrei86} it shown that, in the case of Burgers  flux $f(u)=\frac12\,u^2$, the eigenvalues 
of $\mathcal{L}_\varepsilon$ with homogeneous Dirichlet boundary conditions, are 
real and negative.
Moreover, as a consequence of the requiremente $f(u_+)=f(u_-)$, there holds as $\varepsilon\to 0$
\begin{equation*}
	\lambda_1^\varepsilon=O(e^{-1/\varepsilon})\quad \textrm{and}\quad
			\lambda_k^\varepsilon<-\frac{c_0}{\varepsilon}<0\qquad\forall\,k\geq 2
\end{equation*}
for some $c_0>0$ independent on $\varepsilon$.
Negativity of the eigenvalues implies that the steady state $\bar U_{{}_{\textrm{par}}}^\varepsilon$ is 
asymptotically stable with exponential rate; the precise description of the eigenvalue distribution
shows that the large time behavior is described by term of the order $e^{\lambda_1^\varepsilon\,t}$
and thus  the convergence is very slow when $\varepsilon$ is small
To quantify the reduction order of the mapping $\varepsilon\to e^{-1/\varepsilon}$,
note that $e^{-1/\varepsilon}$ has order $10^{-5}$ for $\varepsilon=10^{-1}$
and order $10^{-44}$ for $\varepsilon=10^{-2}$.

Such is the picture relative to the behavior determined by an initial data close to the 
equilibrium solution $\bar U_{{}_{\textrm{par}}}^\varepsilon$.
The next question concerns with the dynamics generated by initial data presenting a
sharp transition from $u^-$ to $u^+$ localized far from the position of the steady 
state $\bar U_{{}_{\textrm{par}}}^\varepsilon$.
\begin{figure}[ht]\label{fig1}\begin{center}
\includegraphics[width=9cm]{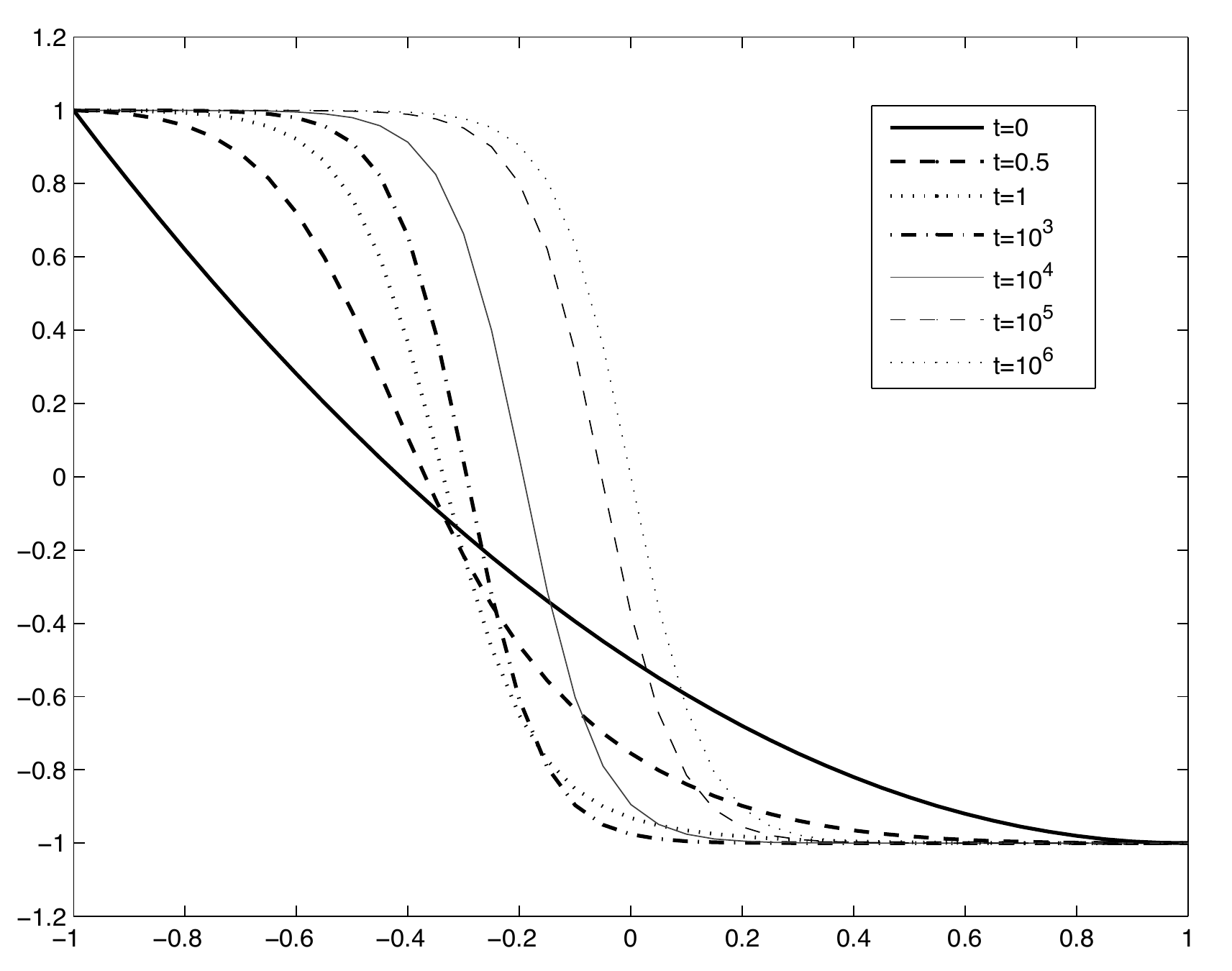}
\caption{\footnotesize The solution to \eqref{burgers}--\eqref{bdary}
with $\varepsilon=0.07$, $u_\pm=\mp 1$ and $u_0(x)=(x^2-2x-1)/2$.}
\end{center}\end{figure}
Figure \ref{fig1} represents a numerical simulation of the solution to the initial value
problem \eqref{burgers} with boundary conditions \eqref{bdary}, relative to the initial
condition $u_0(x)=(x^2-2x-1)/2$.
Starting with a decreasing initial datum, a shock layer is formed in a short time scale, 
so that the solution is approximately given by a translation of the (unique) stationary
solution of the problem.
Once such a layer is formed, on a longer time scale, it moves towards the location
corresponding to the equilibrium solution.

This article deals with the dynamics after the shock layer formation for $\varepsilon$ small.
In order to provide a detailed description of such regime, with special attention to the relation
between the unviscous and the low-viscosity behavior,  it is rational:\\
\indent -- to build up a one-parameter family of functions $\{U^\varepsilon_{{}_{\textrm{par}}}(\cdot;\xi)\}$ 
such that $U^\varepsilon_{{}_{\textrm{par}}}(\cdot;\xi)\to U^\varepsilon_{{}_{\textrm{hyp}}}(\cdot;\xi)$ as
$\varepsilon\to 0$, in an appropriate sense;\\
\indent -- to describe the dynamics of the solution to the initial-boundary value problem
\eqref{burgers}--\eqref{bdary} in a tubular neighborhood of the family
$\{U^\varepsilon_{{}_{\textrm{par}}}(\cdot;\xi)\}$.

A specific element $U^\varepsilon_{{}_{\textrm{par}}}(\cdot;\bar \xi)$ of the manifold
$\{U^\varepsilon_{{}_{\textrm{par}}}\}$ corresponds to the steady state
$\bar U_{{}_{\textrm{par}}}^\varepsilon$ of \eqref{burgers}--\eqref{bdary} and the dynamics
will asymptotically lead to such configuration.
\vskip.15cm

Before describing in details the contribution of the paper, let us recast the state of the art
on the topic.
Among others, the problem of slow dynamics for the Burgers equation
has been examined in \cite{ReynWard95b} and in \cite{LafoOMal95a}, 
where different approaches have been considered. 
The former is based either on {\it projection method} or on {\it WKB expansions}; the latter
stands on an adapted version of the {\it method of matched asymptotics expansion}.
The common aim is to determine an expression and/or an equation for the parameter $\xi$, considered
as a function of time, describing the movement of the transition from a 
generic point of the interval $(-\ell,\ell)$ toward the equilibrium location $\bar\xi$.
In both the contributions, the analysis is conducted at a formal level and validated numerically 
by means of comparision with significant computations.
A rigorous analysis has been performed in \cite{deGrKara98} (and generalized to the case
of nonconvex flux in \cite{deGrKara01}), where  one-parameter family of reference functions
is chosen as a family of traveling wave solutions to the viscous equation satisfying the boundary
conditions and with non-zero (but small) velocity. 
The approach is based on the use of such traveling waves to obtain upper and lower estimates
by the maximum principle, from which rigorous asymptotic formulae for the slow velocity
are obtained.

Slow motion for the viscous Burgers equation in unbounded domains has been also 
considered in literature.
In \cite{Shih95}, it is analyzed the case of the half-line $(0,+\infty)$
for the space variable $x$, with constant initial and boundary data chosen so that speed 
of the shock generated at $x=0$ is stationary for the corresponding hyperbolic equation. 
The presence of the viscosity generates a motion of the transition layer, which is precisely 
identified  by means of the Lambert's $W$ function.
Later, the (slow) motion of a shock wave, with zero hyperbolic speed, for the Burgers equation in 
the quarter plane has been considered in \cite{LiuYu97}, where it is shown that the 
location of the wave front is of order $\ln(1+t)$; the same result has been generalized in \cite{Nish01} 
in the case of general fluxes (for other contributions to the same problem, we refer also to 
\cite{LLLY04, WardReyn95}).

The case of the whole real line has been examined in \cite{KimTzav01} with emphasis on 
the generation of $N-$wave like structures and their evolution towards nonlinear diffusion waves.
The analysis is based on the use of self-similar variables, suggested by the invariance of the 
Burgers equation under the group of transformations $(x,t,u)\mapsto (cx,c^2\,t,u/c)$
(for subsequent contributions in the same direction, see \cite{KimNi02}).
More recently, it has been shown in \cite{BeckWayn09} that the slow motion is determined by the presence of a one-dimensional center manifold of 
steady states for the equation in the self-similar variables (corresponding to the diffusion waves) and 
a relative family of one-dimensional global attractive invariant manifold.
In a short-time scale, the solution approaches one of the attractive manifolds and remains close to
it in a long-time scale. 

At the present day, results relative to metastability in the case of systems appear to be rare. 
Slow dynamics analysis for systems of conservation laws have been considered in \cite{HubeSerr96},
basic model examples being the Navier-Stokes equations of compressible viscous heat 
conductive fluid and the Keyfitz-Kranzer system, arising in elasticity.
The approach is based on asymptotic expansions and consists in deriving appropriate limiting
equations for the leading order terms, in the case of periodic data.
In \cite{KreiKreiLore08}, the problem of proving convergence to a stationary solution for a system of
conservation laws with viscosity is addressed, with an approach based on a detailed
analysis of the linearized operator at the steady state.
A recent contribution is the reference \cite{BJRZ11}, where the authors consider the
Saint-Venant equations for shallow water and, precisely, the phenomenon of formation of roll-waves.
The approach merges together analytical techniques and numerical results to present some intriguing
properties relative to the dynamics of solitary wave pulses.
\vskip.15cm

Summing up, apart for the formal expansions methods, the rigorous approaches used 
in the literature are largely based on typical scalar equations features.
The first of these properties is the direct link between the scalar Burgers equation and the 
heat equation given by the Hopf--Cole transformation: $u=-2\varepsilon\,\phi^{-1}\partial_x \phi$, 
and the consequent invariance of the Burgers equation under the group of 
scaling transformations $(x,t,u)\mapsto (cx,c^2\,t,u/c)$.
On the one hand, the presence of such a connection is an evident advantage, since it permits to determine 
optimal descriptions for the behavior under study (see \cite{KimTzav01, LiuYu97, Shih95});
on the other hand, to use such exceptional property makes the approach very stiff
and difficult to apply to more general cases.
A different ``scalar hallmark'' is the base of the approach considered in \cite{deGrKara98},
where the authors make wide use of maximum principle and comparison arguments,
taking benefit from the fact that the equation is second-order parabolic.

In order to extend the results to more general settings and specifically for systems of PDEs,
it is useful to determine strategies and techniques that are more flexible, paying, if necessary, 
the price of a less accurate description of the dynamics.
A contribution in this direction has been given in \cite{Nish01}, where the location of 
the shock transition for a scalar conservation law in the quarter plane has been proved by means
of weighted energy estimates, extending the result proved in \cite{LiuYu97}, that used
an explicit formula --determined by means of the Hopf--Cole transformation-- for the Green function 
of the linearization at the shock profile of the Burgers equation.
\vskip.15cm

The present article intends to contribute to the definition of a versatile language for metastability,
suitable for general class of partial differential equations of evolutive type.
With this direction in mind, we follow an approach that it is strictly related with the 
{\it projection method} considered in \cite{CarrPego89, ReynWard95b} and we go behind 
the philosophy tracked in the analysis of stability of viscous shock waves by K.Zumbrun 
and co-authors  (see \cite{ZumbHowa98, MascZumb03, MascZumb04}).
Precisely, we separate three distinct phases:\par
\indent {\bf i.} to choose a family of functions $\{U^\varepsilon(\cdot;\xi)\}$,
considered as approximate solutions, and to measure how far they are from being exact solutions;\\
\indent {\bf ii.} to investigate spectral properties of the linearized operators at such states; \\
\indent {\bf iii.} to show that appropriate assumptions on the approximate solutions (step {\bf i}) 
and on the spectrum of the linearized operators (step {\bf ii})
imply the appearance of a metastable behavior.
\vskip.15cm

With respect to the framework of shock waves stability analysis, there are two main differences.
First of all, we concentrate on the case of bounded domains and, therefore, the spectrum of the
linearized operators is discrete.
Additionally, since the reference states $U^\varepsilon$ are approximate solutions, the perturbations
of such states satisfy at first order a {\it non-homogeneous} linear equation, with forcing term
negligible as $\varepsilon\to 0^+$.
The defect of working in a neighborhood of a manifold that is not invariant has the counterpart of a
wider flexibility in its construction that leads, in particular, to (more or less) explicit representations.
Thus, it should be possible in principle to obtain numerical evidence of special spectral properties
even in cases where analytical results appear to be not achievable.

The article is organized as follows.
To start with, in Section \ref{Sect:Abs},
we consider a general framework containing scalar viscous conservation laws as a very specific case. 
Given a family of approximate solutions $\{U^\varepsilon\}$, our approach consists in representing
the solution to the initial-bondary value problem as the sum of an element $U^\varepsilon(\cdot;\xi(t))$
moving along the family $\{U^\varepsilon\}$ plus a perturbation term $v$.
The equation for the unknown $\xi=\xi(t)$ is chosen in such a way that the slower decaying terms
in the perturbation $v$ are canceled out.
In order to state a general result, we consider an approximation of the complete nonlinear equations
for the couple $(\xi, v)$, obtained by disregarding quadratic terms in $v$ and keeping the nonlinear
dependence on $\xi$,
in order to keep track of the nonlinear evolution along the manifold $\{U^\varepsilon\}$.
Such reduced system for $(v,\xi)$ is called {\sf quasi-linearized system} and it is the
concern of Theorem \ref{thm:metaL}, the main contribution of the paper.
Under appropriate assumptions on the manifold $U^\varepsilon$,
the linearized operators at such states, and the coupling between the two objects, 
such result gives an explicit representation for the solution to the evolutive problem 
together with an estimate on the remainder, vanishing in the limit $\varepsilon\to 0$.
This gives a sound justification to the reduced equation for the unknown $\xi=\xi(t)$ 
obtainable by neglecting also the linear term in $v$.

Dealing with the complete system for the couple $(v,\xi)$ brings into the analysis also the
specific form of the quadratic terms. 
As a consequence, in case of parabolic systems of reaction-diffusion type, we expect that
a results analogous to Theorem \ref{thm:metaL} could be proved, under the assumption
of an {\it a priori} $L^\infty$ bound on the solution.
Differently, when a nonlinear first order space derivative term is present (as is the
case for viscous conservation laws), the quadratic term involve a dependence on the
space derivative of the solution and a rigorous result needs an additional bound, which
we are not presently able to achieve.

In Section \ref{Sect:Appl} we consider the application of the general framework to 
the case of viscous scalar conservation laws. 
Firstly, we present the dynamics of the hyperbolic equation obtained in the vanishing
viscosity limit, proving a result on finite-time convergence to the one-parameter manifold
of steady states (Theorem \ref{thm:stabhyp}).
Then, we pass to consider the parabolic equation in \eqref{burgers} under assumption
\eqref{fluxhyp} and we build up a specific family $\{U^\varepsilon\}$ by matching continuously
stationary solutions at a given point $\xi$.
To apply the general result of Section \ref{Sect:Abs}, we need to measure how far are
states $U^\varepsilon$ from being stationary solutions, and this amounts in estimating
the jump of the space derivative at the matching point.
Such task is completed, showing that the residual has order $Ce^{-C/\varepsilon}$,
hence it is exponentially small in the limit $\varepsilon\to 0^+$.
As a by-product, we deduce a formal equation for the motion of the shock layer, which 
generalizes the one known for the case of the Burgers flux $f(s)=\frac12s^2$.

In Section \ref{Sect:Spec}, we analyze spectral properties of the diffusion-transport linear operator,
arising from the linearization at the state $U^\varepsilon(\cdot;\xi)$. 
We show that, under appropriate assumption on the limiting behavior of $U^\varepsilon$ as $\varepsilon\to 0^+$,
the spectrum can be decomposed into two parts: the first eigenvalue of order $O(e^{-C/\varepsilon})$;
all of the remaining eigenvalues are less than $-C/\varepsilon$
(where $C$ denotes a generic positive constant independent on $\varepsilon$).
Additionally, precise asymptotics for the first eigenvalue are achieved by considering the linear operator
with piecewise constant coefficient, obtained by taking the limit of functions $U^\varepsilon(\cdot;\xi)$
as $\varepsilon\to 0^+$.
This analysis is needed to give evidence of the validity of the coupling assumption required
in Theorem \ref{thm:metaL}.

\section{Metastable behavior for nonlinear parabolic systems}\label{Sect:Abs}

Given $\ell>0$, $I:=(-\ell,\ell)$ and $n\in\N$, 
we consider the space $X:=[L^2(I)]^n$  endowed with
\begin{equation*}
	\langle u, v\rangle:=\int_{-\ell}^{\ell} u(x)\cdot v(x)\,dx
	\qquad\qquad u,v\in X,
\end{equation*}
where $\cdot$ denotes the usual scalar product in $\R^n$.
Given $T>0$, we consider the evolutive Cauchy problem for the unknown $u\,:\,[0,T)\to X$
\begin{equation}\label{cauchy}
	\partial_t u={\mathcal F}^\varepsilon[u],\qquad	u\,\bigr|_{t=0}=u_0
\end{equation}
where ${\mathcal F}^\varepsilon$ denotes a nonlinear differential operator,
complemented with appropriate boundary conditions.
We are interested in describing the dynamical behavior of $u^{\varepsilon}$, solution to \eqref{cauchy},
in the regime $\varepsilon\sim 0$.
In particular, we have in mind the case of a singular dependence of ${\mathcal F}^\varepsilon$
with respect to $\varepsilon$, in the sense that the operator ${\mathcal F}^0$
is of lower order with respect to ${\mathcal F}^\varepsilon$.
The specific example, considered in detail in the subsequent Sections,
is the one-dimensional scalar viscous conservation laws with Dirichlet boundary conditions;
at the same time, also the usual Allen--Cahn parabolic equation fits into the framework.

Given a one-dimensional open interval $J$, let $\{U^{\varepsilon}(\cdot;\xi)\,:\,\xi\in J\}$ be a
one-parameter family in $X$, whose elements can be considered as approximate stationary
solutions to the problem in the sense that ${\mathcal F}^\varepsilon[U^{\varepsilon}(\cdot;\xi)]$
depends smoothly on $\varepsilon$ and tends to $0$ as $\varepsilon\to 0$. Precisely, we assume that the term ${\mathcal F}^\varepsilon[U^{\varepsilon}]$ belongs 
to the dual space of the continuous functions space $C(I)$ and there exists a family of smooth
positive  functions $\Omega^\varepsilon=\Omega^\varepsilon(\xi)$, uniformly convergent
to zero as $\varepsilon\to 0$,  such that, for any $\xi\in J$, there holds
\begin{equation}\label{defOmegaeps}
	|\langle \psi(\cdot),{\mathcal F}^\varepsilon[U^{\varepsilon}(\cdot,\xi)]\rangle|
		\leq \Omega^\varepsilon(\xi)\,|\psi|_{{}_{\infty}} 
		\qquad \forall\,\psi\in C(I).
\end{equation}
The family $\{U^{\varepsilon}(\cdot;\xi)\}$ will be referred to as an {\sf approximate invariant manifold}
with respect to the flow determined by \eqref{cauchy} in $X$.
Generically, since an element $U^{\varepsilon}(\cdot;\xi)$ is not a steady state for \eqref{cauchy},
the dynamics walk away from the manifold with a speed dictated by $\Omega^\varepsilon$.
The dependence of $\Omega^\varepsilon$ on $\varepsilon$ plays a relevant r\^ole,
since it drives the departure from the approximate  invariant manifold. 

Next, we decompose the solution to the initial value problem \eqref{cauchy} as 
\begin{equation*}
	u(\cdot,t)=U^{\varepsilon}(\cdot;\xi(t))+v(\cdot,t)
\end{equation*} 
with $\xi=\xi(t)\in J$  and $v=v(\cdot,t)\in [L^2(I)]^n$ to be determined.
Substituting, we obtain
\begin{equation}\label{eqv}
	\partial_t v=\mathcal{L}^\varepsilon_\xi v+{\mathcal F}^\varepsilon[U^{\varepsilon}(\cdot;\xi)]
		-\partial_{\xi}U^{\varepsilon}(\cdot;\xi)\,\frac{d\xi}{dt}+\mathcal{Q}^\varepsilon[v,\xi]
\end{equation}
where
\begin{equation*}
	\begin{aligned}
		\mathcal{L}^\varepsilon_\xi v&:=d{\mathcal F}^\varepsilon[U^{\varepsilon}(\cdot;\xi)]\,v\\
		\mathcal{Q}^\varepsilon[v,\xi]&:={\mathcal F}^\varepsilon[U^{\varepsilon}(\cdot;\xi)+v]
			-{\mathcal F}^\varepsilon[U^{\varepsilon}(\cdot;\xi)]
				-d{\mathcal F}^\varepsilon[U^{\varepsilon}(\cdot;\xi)]\,v.
	\end{aligned} 
\end{equation*}
Next, we assume that the linear operator $\mathcal{L}^\varepsilon_\xi$ has a discrete spectrum
composed by semi-simple eigenvalues $\lambda^\varepsilon_k=\lambda^\varepsilon_k(\xi)$
with corresponding right eigenfunctions $\phi^\varepsilon_k=\phi^\varepsilon_k(\cdot;\xi)$.
Denoting by $\psi^\varepsilon_k=\psi^\varepsilon_k(\cdot;\xi)$ the eigenfunctions of the adjoint operator 
$\mathcal{L}^{\varepsilon,\ast}_\xi$ and setting 
\begin{equation*}
	v_k=v_k(\xi;t):=\langle \psi^\varepsilon_k(\cdot;\xi),v(\cdot,t)\rangle, 
\end{equation*}
we can use the degree of freedom we still have in the choice of the couple $(v,\xi)$ in such a way that 
the component $v_1$  is identically zero, that is
\begin{equation*}
	\frac{d}{dt} \langle \psi^\varepsilon_1(\cdot;\xi(t)), v(\cdot,t) \rangle =0
	\qquad\textrm{and}\qquad
	\langle \psi^\varepsilon_1(\cdot;\xi_0), v_0(\cdot))\rangle=0.
\end{equation*}
Using equation \eqref{eqv}, we infer
\begin{equation*}
	\langle \psi^\varepsilon_1(\xi,\cdot),\mathcal{L}^\varepsilon_\xi v+ {\mathcal F}[U^{\varepsilon}(\cdot;\xi)]
		- \partial_{\xi}U^{\varepsilon}(\cdot;\xi)\frac{d\xi}{dt}+{\mathcal Q}^\varepsilon[v,\xi] \rangle 
		+ \langle  \partial_{\xi}\psi^\varepsilon_1(\xi,\cdot) \frac{d\xi}{dt},v \rangle =0
\end{equation*}
Since $\langle \psi^\varepsilon_1, {\mathcal L}_{\xi}v \rangle= \lambda_1\langle \psi^\varepsilon_1, v \rangle$, 
we obtain a scalar differential equation for the variable $\xi$, describing the reduced
dynamics along the approximate manifold, that is
\begin{equation}\label{eqxi0}
		\alpha^\varepsilon(\xi,v)\frac{d\xi}{dt}=\langle \psi^\varepsilon_1(\cdot;\xi),
			{\mathcal F}[U^{\varepsilon}(\cdot;\xi)]+\mathcal{Q}^\varepsilon[v,\xi] \rangle
\end{equation}
where
\begin{equation*}	
	\alpha^\varepsilon_0(\xi)
			:= \langle \psi^\varepsilon_1(\cdot;\xi), \partial_{\xi}U^{\varepsilon}(\cdot;\xi) \rangle
	\quad\textrm{and}\quad
	\alpha^\varepsilon(\xi,v)
			:= \alpha^\varepsilon_0(\xi) - \langle  \partial_{\xi}\psi^\varepsilon_1(\cdot;\xi),v \rangle,
\end{equation*}
together with the condition on the initial datum $\xi_0$
\begin{equation*}
	\langle \psi^\varepsilon_1(\cdot;\xi_0), v_0(\cdot) \rangle =0
\end{equation*}
To rewrite equation \eqref{eqxi0} in normal form in the regime of small $v$, 
we assume
\begin{equation*}
	|\alpha^\varepsilon_0(\xi)|
	=|\langle \psi^\varepsilon_1(\cdot;\xi), \partial_{\xi}U^{\varepsilon}(\cdot;\xi) \rangle|
	\geq c_0>0
\end{equation*}
for some $c_0>0$ independent on $\xi$.
Such assumption gives a (weak) restriction on the choice of the members of the family
$\{U^{\varepsilon}\}$ asking for the manifold to be never transversal to the first eigenfunction
of the corresponding linearized operator.
From now on, we can renormalize the eigenfunction 
$\psi^\varepsilon_1$ so that
\begin{equation*}
	\alpha^\varepsilon_0(\xi)
	=\langle \psi^\varepsilon_1(\cdot;\xi), \partial_{\xi}U^{\varepsilon}(\cdot;\xi)\rangle=1,
\end{equation*}
for any $\varepsilon>0$ and for any $\xi\in J$.
In the regime $v\to 0$, we may expand $1/\alpha^\varepsilon$ as
\begin{equation*}
	\frac{1}{\alpha^\varepsilon(\xi,v)} = \frac{1}{\alpha^\varepsilon_0(\xi)}\left(1
		+\frac{\langle \partial_{\xi} \psi^\varepsilon_1,v \rangle}{\alpha^\varepsilon_0(\xi)}\right) + o(|v|)
		= 1+\langle \partial_{\xi} \psi^\varepsilon_1,v \rangle+ o(|v|).
\end{equation*}
Inserting in \eqref{eqxi0}, we mat rewrite the nonlinear equation for $\xi$ as
\begin{equation}\label{eqxiNL}
	\frac{d\xi}{dt}=\theta^\varepsilon(\xi)\bigl(1+\langle\partial_{\xi} \psi^\varepsilon_1, v \rangle\bigr)
		+ \rho^\varepsilon[\xi,v], 
\end{equation}
where
\begin{equation*}
	\begin{aligned}
 	\theta^\varepsilon(\xi)
		&:=\langle \psi^\varepsilon_1,{\mathcal F[U^{\varepsilon}] \rangle}\\
	\rho^\varepsilon[\xi,v]&:=\frac{1}{\alpha^\varepsilon(\xi,v)}
		\bigl(\langle \psi^\varepsilon_1,\mathcal{Q}^\varepsilon\rangle
		+\langle \partial_\xi\psi^\varepsilon_1, v\rangle^2\bigr).
	\end{aligned}
\end{equation*}
Using \eqref{eqxiNL}, equation  \eqref{eqv} can be rephrased as
\begin{equation}\label{eqvNL}
	\partial_t v= H^\varepsilon(x;\xi)
		+ ({\mathcal L}^\varepsilon_\xi+{\mathcal M}^\varepsilon_\xi)v
			+\mathcal{R}^\varepsilon[v,\xi]
\end{equation}
where 
\begin{align*}
		H^\varepsilon(\cdot;\xi)&:={\mathcal F}^\varepsilon[U^{\varepsilon}(\cdot;\xi)]
			-\partial_{\xi}U^{\varepsilon}(\cdot;\xi)\,\theta^\varepsilon(\xi),\\
		{\mathcal M}^\varepsilon_\xi v&:=-\partial_{\xi}U^{\varepsilon}(\cdot;\xi)
			\,\theta^\varepsilon(\xi)\,\langle\partial_{\xi} \psi^\varepsilon_1, v \rangle,\\ 
		\mathcal{R}^\varepsilon[v,\xi]&:=\mathcal{Q}^\varepsilon[v,\xi]
								-\partial_{\xi}U^{\varepsilon}(\cdot;\xi)\,\rho^\varepsilon[\xi,v].
\end{align*}
Let us stress that, by definition, there holds
\begin{equation*}
	\langle \psi^\varepsilon_1(\cdot;\xi), H^\varepsilon(\cdot;\xi)\rangle=0,
\end{equation*}
so that $H^\varepsilon(\cdot;\xi)$ is the projection of ${\mathcal F}^\varepsilon[U^{\varepsilon}(\cdot;\xi)]$
onto the space orthogonal to $\phi^\varepsilon_1(\cdot;\xi)$.

Summarizing, the couple $(v,\xi)$ solves the differential system \eqref{eqxiNL}-\eqref{eqvNL}
where the initial condition $\xi_0$ for $\xi$ is such that
\begin{equation*}
	\langle \psi^\varepsilon_1(\cdot;\xi_0), u_0-U(\cdot;\xi_0)\rangle =0
\end{equation*}
and the initial condition $v_0$ for $v$ is given by $u_0-U(\cdot;\xi_0)$.

Neglecting the $o(v)$ order terms, we obtain the system
\begin{equation}\label{LS}
 	\left\{\begin{aligned}
	\frac{d\zeta}{dt}&=\theta^\varepsilon(\zeta)\bigl(1 
		+\langle\partial_{\zeta} \psi^\varepsilon_1, w \rangle\bigr), \\
	\partial_t w &= H^\varepsilon(\zeta)+ ({\mathcal L}^\varepsilon_\zeta+{\mathcal M}^\varepsilon_\zeta)w
 	\end{aligned}\right. 
\end{equation}
with initial conditions
\begin{equation}\label{initialLS}
	\zeta(0)=\zeta_0\in(-\ell,\ell) \qquad\textrm{and}\qquad w(x,0)=w_0(x)\in X.
\end{equation}
From now on, we will refer to this system as the {\sf quasi-linearization} of \eqref{eqxiNL}--\eqref{eqvNL}.
Our aim is to describe the behavior of the solution to \eqref{LS} in the regime
of small $\varepsilon$.

Shortly, the quasi-linearized system is determined by an appropriate combination of the 
term ${\mathcal F}^\varepsilon[U^{\varepsilon}]$, measuring how far is the function $U^\varepsilon$
from being a stationary solution, and the linear operator $\mathcal L_\xi^\varepsilon$, controlling
at first order how  solutions to \eqref{cauchy} depart from $U^\varepsilon$ when the latter is taken
as initial datum.
To state our first result, we need to precise the assumption on such terms.
\vskip.15cm

{\bf H1.} The family $\{U^{\varepsilon}(\cdot,\xi)\}$ is such that ${\mathcal F}^\varepsilon[U^{\varepsilon}]$
belongs to the dual space of $C(I)^n$ and there exists functions $\Omega^\varepsilon$ such that,
denoting again with $\langle \cdot ,\cdot\rangle$ the duality relation, 
\begin{equation*}
	|\langle \psi(\cdot),{\mathcal F}^\varepsilon[U^{\varepsilon}(\cdot,\xi)]\rangle|
		\leq \Omega^\varepsilon(\xi)\,|\psi|_{{}_{\infty}} \qquad \forall\,\psi\in C(I).
\end{equation*}
with $\Omega^\varepsilon$ converging to zero as $\varepsilon\to 0$, uniformly with respect to $\xi\in J$.
\vskip.15cm

{\bf H2.} The eigenvalues $\{\lambda^\varepsilon_k(\xi)\}_{{}_{k\in\N}}$ of $\mathcal{L}^\varepsilon_\xi$
are semi-simple, $\lambda_1(\xi)$ is simple, real and negative, and
\begin{equation*}
	\textrm{\rm Re}\,\lambda^\varepsilon_k(\xi)\leq \min\{ \lambda^\varepsilon_1(\xi)-C,-C\,k^2\}
	\quad \textrm{for }k\geq 2.
\end{equation*}
for some constant $C>0$ independent on $k\in\N$, $\varepsilon>0$ and $\xi\in J$.
\vskip.15cm

{\bf H3.} The eigenfunctions $\phi^\varepsilon_k(\cdot;\xi)$ and $\psi^\varepsilon_k(\cdot;\xi)$
of $\mathcal{L}^\varepsilon_{\xi}$ and $\mathcal{L}^{\varepsilon,\ast}_{\xi}$
normalized so that
\begin{equation*}
	\langle \psi^\varepsilon_1(\cdot;\xi), \partial_{\xi}U^{\varepsilon}(\cdot;\xi)\rangle=1
	\qquad\textrm{and}\qquad
	\langle \psi^\varepsilon_j, \phi^\varepsilon_k\rangle=\delta_{jk}.
\end{equation*}
where $\delta_{jk}$ is the usual Kronecker symbol, are such that
\begin{equation}\label{derpsiphi}
	\sum_{j} \langle \partial_\xi \psi^\varepsilon_k, \phi^\varepsilon_j\rangle^2
	=\sum_{j} \langle \psi^\varepsilon_k, \partial_\xi \phi^\varepsilon_j\rangle^2
	\leq C
	\qquad\qquad\forall\,k.
\end{equation}
for some constant $C$ independent on $\varepsilon>0$ and $\xi\in J$.

The last assumption we require relate the term $\Omega^\varepsilon(\xi)$ to the first eigenvalue
$\lambda_1^\varepsilon(\xi)$ of the  linearized operator $\mathcal{L}_\xi^\varepsilon$
at $U^\varepsilon(\cdot;\xi)$.
Formally, if $U^\varepsilon(\cdot;\bar\xi)$ is an exact stationary solution, then
\begin{equation*}
	\mathcal{F}[U^\varepsilon(\cdot;\xi)]
	=\mathcal{F}[U^\varepsilon(\cdot;\xi)]-\mathcal{F}[U^\varepsilon(\cdot;\bar\xi)]
	\approx\mathcal{L}_\xi^\varepsilon\partial_\xi U^\varepsilon(\cdot;\bar\xi)(\bar\xi-\xi).
\end{equation*}
If $\partial_\xi U^\varepsilon$ is chosen to be approximately close to the 
first eigenfunction of $\mathcal{L}_\xi^\varepsilon$, then
\begin{equation*}
	\langle \psi(\cdot),\mathcal{F}[U^\varepsilon(\cdot;\xi)]\rangle 
	=\mathcal{F}[U^\varepsilon(\cdot;\xi)]-\mathcal{F}[U^\varepsilon(\cdot;\bar\xi)]
	\approx \lambda_1^\varepsilon(\xi)\langle \psi(\cdot),\partial_\xi U^\varepsilon(\cdot;\bar\xi)\rangle (\bar\xi-\xi),
\end{equation*}
so that, heuristically, there exists a constant $C>0$ such that
\begin{equation*}
	|\langle \psi(\cdot),\mathcal{F}[U^\varepsilon(\cdot;\xi)]\rangle|
	\leq C|\lambda_1^\varepsilon(\xi)||\psi|_{{}_{\infty}}
\end{equation*}
which gives the final form of our ultimate assumption.

\begin{theorem}\label{thm:metaL}
Let hypotheses {\bf H1-2-3} be satisfied.
Additionally, assume that 
\begin{equation}\label{meta}
	\Omega^\varepsilon(\xi)\leq C|\lambda_1^\varepsilon(\xi)|
\end{equation}
for some constant $C>0$ independent on $\varepsilon>0$ and $\xi\in J$.

Then, denoted by $(\zeta,w)$ the solution to the initial-value problem \eqref{LS}--\eqref{initialLS},
for any $\varepsilon$ sufficiently small, there exists a time $T^\varepsilon$ 
such that for any $t\leq T^\varepsilon$ the solution $w$ is given by
\begin{equation*}
 w=z+R
\end{equation*}
where $z$ is defined by
\begin{equation*}
	z(x,t):=\sum_{k\geq 2} w_k(0)\exp\left(\int_0^t \lambda^\varepsilon_k(\zeta(\sigma))\,d\sigma\right)
		\,\phi^\varepsilon_k(x;\zeta(t)),
\end{equation*}
and the remainder $R$ satisfies the estimate
\begin{equation}\label{boundw}
	|R|_{{}_{L^2}}\,\leq C\,|\Omega^\varepsilon|_{{}_{\infty}}
		\left\{\exp\left(2\int_0^t\lambda_1^\varepsilon(\zeta(\sigma))d\sigma \right)
			|w_0|_{{}_{L^2}}^2+1\right\}
\end{equation}
for some constant $C>0$ independent on $\varepsilon, T>0$.

Moreover, for initial data $w_0$ sufficiently small in $L^2$, the final time $T^\varepsilon$
can be chosen with order $\left|\ln |\Omega^\varepsilon|_{{}_{\infty}}\right|/|\Omega^\varepsilon|_{{}_{\infty}}$.
\end{theorem}
\vskip.15cm

The conclusion of the proof of Theorem \ref{thm:metaL} is based on the following 
version of a standard nonlinear iteration argument.
\vskip.15cm

\begin{lemma}\label{lem:nquadro}
Let $f=f(t), g=g(t)$ and $h=h(s,t)$ be continuous functions for $t\in[0,T]$
for some $T>0$, such that
\begin{equation*}
	f(t)\geq 0,\quad g(t)>0,\quad g\textrm{ decreasing},\quad h(s,t)\geq 0.
\end{equation*}
Let $y=y(t)$ be a non-negative function satisfying the estimate
\begin{equation*}
	y(t)\leq \int_0^t \left\{f(s)\,g(t)\,y^2(s)+h(s,t)\right\}\,ds
\end{equation*}
for any $t\leq T$. 
If there holds
\begin{equation}\label{AB14}
    \sup_{t\in[0,T]} \int_0^t g^2(s)\,f(s)\,ds  \,\cdot\,
    \sup_{t\in[0,T]} \frac{1}{g(t)}\,\int_0^t h(s,t)\,ds
    < \frac14
\end{equation}
for any $t\in[0,T]$, then
\begin{equation*}
	y(t)\leq 2\,\sup_{\tau\in[0,t]} \int_0^\tau h(s,\tau)\,ds
\end{equation*}
for any $t\in[0,T]$.
\end{lemma}

\begin{proof}[Proof of Lemma \ref{lem:nquadro}]
The auxiliary function $w(t):=g^{-1}(t)\,y(t)$ enjoyes the estimate
\begin{equation*}
	w(t)\leq \int_0^t \left\{\alpha(s)\,w^2(s)+\beta(s,t)\right\}\,ds
\end{equation*}
where $\alpha(t):=f(t)\,g^2(t)$ and $\beta(s,t)=g^{-1}(t)\,h(s,t)$.
The quantity 
\begin{equation*}
	N(t):=\sup_{\tau\in[0,t]} w(\tau).
\end{equation*}
is such that for any $t\in[0,T]$ there holds
\begin{equation*}
	N(t)\leq A\,N^2(t)+B
\end{equation*}
where 
\begin{equation*}
	A=A(T):=\sup_{t\in[0,T]} \int_0^t \alpha(s)\,ds,\qquad
	B=B(T):=\sup_{t\in[0,T]} \int_0^t \beta(s,t)\,ds.
\end{equation*}
Since $N(0)=0$, if $1-4AB>0$, then
\begin{equation*}
	N<\frac{1-\sqrt{1-4AB}}{2A}=\frac{2B}{1+\sqrt{1-4AB}}\leq 2B.
\end{equation*}
In term of $y$, if \eqref{AB14} holds,  then
\begin{equation*}
 y(t)<2\,g(t)\,\sup_{\tau\in[0,T]} \frac{1}{g(\tau)}\int_0^\tau h(s,\tau)\,ds.
\end{equation*}
The final estimate follows from the monotonicity of the function $g$.
\end{proof}

\begin{proof}[Proof of Theorem \ref{thm:metaL}]
Setting
\begin{equation*}
	w(x,t)=\sum_{j} w_j(t)\,\phi^\varepsilon_j(x,\zeta(t)),
\end{equation*}
we obtain an infinite-dimensional differential system for the coefficients $w_j$
\begin{equation}\label{eqwk_bis}
	\frac{dw_k}{dt}=\lambda^\varepsilon_k(\zeta)\,w_k
		+\langle \psi^\varepsilon_k,F\rangle
\end{equation}
where, omitting the dependencies for shortness,
\begin{equation*}
	F:=H^\varepsilon+\sum_{j} w_j\,\Bigl\{{\mathcal M}^\varepsilon_\zeta\, \phi^\varepsilon_j
			-\partial_\xi \phi^\varepsilon_j\,\frac{d\zeta}{dt}\Bigr\}
		=H^\varepsilon-\theta^\varepsilon \sum_{j}\Bigl(a_j+\sum_{\ell} b_{j\ell}\,w_\ell\Bigr)w_j.
\end{equation*}
and the coefficients $a_j$, $b_{jk}$ are given by
\begin{equation*}
	a_j:=\langle \partial_{\xi} \psi^\varepsilon_1, \phi^\varepsilon_j\rangle\,
			\partial_{\xi}U^{\varepsilon}+\partial_\xi \phi^\varepsilon_j,
	\qquad
	b_{j\ell}:=\langle \partial_{\xi} \psi^\varepsilon_1, \phi^\varepsilon_\ell\rangle
			\,\partial_\xi \phi^\varepsilon_j
\end{equation*}
Convergence of the series is guaranteed by assumption \eqref{derpsiphi}.

Differentiating the normalization condition on the eigenfunction, we infer
\begin{equation*}
	\langle \partial_\xi \psi^\varepsilon_j, \phi^\varepsilon_k\rangle
	+\langle \psi^\varepsilon_j,  \partial_\xi \phi^\varepsilon_k\rangle=0.
\end{equation*}
Thus,  for the coefficients $a_j$ there hold
\begin{equation*}
	\langle \psi^\varepsilon_k, a_j\rangle 
		=\langle \partial_{\xi} \psi^\varepsilon_1, \phi^\varepsilon_j\rangle\,
			\bigl(\langle \psi^\varepsilon_k, \partial_{\xi}U^{\varepsilon}\rangle-1\bigr),
\end{equation*}
so that, in particular, $\langle \psi^\varepsilon_1, a_j\rangle=0$ for any $j$.
Thus, equation \eqref{eqwk_bis} for $k=1$ becomes
\begin{equation}\label{eqw1}
	\frac{dw_1}{dt}=\lambda^\varepsilon_1(\zeta)\,w_1
		-\theta^\varepsilon(\zeta) \sum_{\ell, j} \langle \psi^\varepsilon_1, b_{j\ell}\rangle \,w_\ell\,w_j
\end{equation}
Now let us set
\begin{equation*}
	E_k(s,t):=\exp\left( \int_s^t \lambda_k^\varepsilon(\zeta(\sigma))d\sigma\right).
\end{equation*}
As a consequence of hypothesis {\bf H2.}, there exists $C>0$ such that 
$\textrm{\rm Re}\,\lambda_k(\xi)\leq \lambda_1(\xi)-Ck^2$ for any $k\geq 2$.
Thus, the absolute value of $E_k$, $k\geq 2$, can be estimated by
\begin{equation*}
	|E_k|(s,t)\leq \exp\left( \int_s^t \textrm{\rm Re}\, \lambda_k^\varepsilon(\zeta(\sigma))d\sigma\right)
	\leq  E_1(s,t)\,e^{-Ck^2(t-s)}
\end{equation*}
From equalities \eqref{eqw1} and \eqref{eqwk_bis}, choosing $w_1(0)=0$, there follow
\begin{equation*}
	\begin{aligned}
	w_1(t)&=-\int_0^t \theta^\varepsilon(\zeta)
		\sum_{\ell, j} \langle \psi^\varepsilon_1, b_{j\ell}\rangle \,w_\ell\,w_j
		\,E_1(s,t)\,ds\\
	w_k(t)&=w_k(0)\,E_k(0,t)\\
		& +\int_0^t \Bigl\{\langle \psi^\varepsilon_k,H^\varepsilon\rangle
			-\theta^\varepsilon(\zeta)\sum_{j}\Bigl(\langle \psi^\varepsilon_k, a_j\rangle 
			+\sum_{\ell} \langle \psi^\varepsilon_k, b_{j\ell}\rangle \,w_\ell\Bigr) w_j
			\Bigr\} E_k(s,t)\,ds,
	\end{aligned}
\end{equation*}
for $k\geq 2$.
Such expressions suggest to introduce the function 
\begin{equation*}
	z(x,t):=\sum_{k\geq 2} w_k(0)\,E_k(0,t)\,\phi^\varepsilon_k(x;\zeta(t)),
\end{equation*}
From the representation formulas for the coefficients $w_k$, since
\begin{equation*}
 	|\theta^\varepsilon(\zeta)|\leq C\,\Omega^\varepsilon(\zeta)
		\qquad\textrm{and}\qquad
	|\langle \psi^\varepsilon_k,H^\varepsilon\rangle| \leq C\,\Omega^\varepsilon(\zeta)
		\left\{1+ |\langle \psi_k^\varepsilon, \partial_\xi U^\varepsilon\rangle|\right\}
\end{equation*}
for some constant $C>0$ depending on the $L^\infty-$norm of $\psi^\varepsilon_k$,
there holds
\begin{equation*}
	\begin{aligned}
	&|w-z|_{{}_{L^2}}^2\leq C\Bigl(\int_0^t \Omega^\varepsilon(\zeta)
		\sum_{j}|\langle \psi^\varepsilon_1, \,\partial_\xi \phi^\varepsilon_j\rangle|\,|w_j|\,
		\sum_{\ell} |\langle \partial_{\xi} \psi^\varepsilon_1, \phi^\varepsilon_\ell\rangle|\,|w_\ell|
		\,E_1(s,t)\,ds\Bigr)^2\\
	&\qquad +C\sum_{k\geq 2}\Bigl(\int_0^t
		 \Omega^\varepsilon(\zeta)\Bigl(1+ |\langle \psi_k^\varepsilon, \partial_\xi U^\varepsilon\rangle|
		+|\langle \psi^\varepsilon_k,\partial_{\xi}U^{\varepsilon}\rangle|
			\sum_{j}|\langle \partial_{\xi} \psi^\varepsilon_1, \phi^\varepsilon_j\rangle||w_j|\\
	&\qquad
		+\sum_{j}|\langle \partial_\xi \psi^\varepsilon_k,  \phi^\varepsilon_j\rangle||w_j|
		+\sum_{j} |\langle \psi^\varepsilon_k,\partial_\xi \phi^\varepsilon_j\rangle|\,|w_j|
		\sum_{\ell}|\langle \partial_{\xi} \psi^\varepsilon_1, \phi^\varepsilon_\ell\rangle|\,|w_\ell|\Bigr)
			|E_k|(s,t)\Bigr)^2\\
	&\hskip1cm
	\leq C\Bigl(\int_0^t \Omega^\varepsilon(\zeta)|w|_{{}_{L^2}}^2 E_1(s,t)\,ds\Bigr)^2
		+C\sum_{k\geq 2}\Bigl(\int_0^t
		 \Omega^\varepsilon(\zeta)\bigl(1+|w|_{{}_{L^2}}^2\bigr)|E_k|(s,t)\,ds\Bigr)^2
	\end{aligned}
\end{equation*}
Since $\sqrt{a+b}\leq \sqrt{a}+\sqrt{b}$, we infer
\begin{equation*}
	\begin{aligned}
	|w-z|_{{}_{L^2}}&\leq 
		C\int_0^t \Omega^\varepsilon(\zeta)|w|_{{}_{L^2}}^2\,E_1(s,t)\,ds
		+C\sum_{k\geq 2}\int_0^t\Omega^\varepsilon(\zeta)
			\bigl(1+|w|_{{}_{L^2}}^2\bigr)|E_k|(s,t)\,ds\\
				&\leq 
		C\int_0^t \Omega^\varepsilon(\zeta)\Bigl\{|w|_{{}_{L^2}}^2\,E_1(s,t)
			+\bigl(1+|w|_{{}_{L^2}}^2\bigr)\,\sum_{k\geq 2} |E_k|(s,t)\Bigr\}\,ds.
	\end{aligned}
\end{equation*}
The assumption on the asymptotic behavior of the eigenvalues $\lambda_k$
can now be used to bound the series.
Indeed, there holds for some $C>0$
\begin{equation*}
	\sum_{k\geq 2} |E_k(s,t)| \leq \sum_{k\geq 2} E_1(s,t)\,e^{-Ck^2(t-s)}
		\leq C\,E_1(s,t)\,(t-s)^{-1/2}\,e^{-C(t-s)}
\end{equation*}
As a consequence, for unknown $w$ such that $|w|_{{}_{L^2}}\leq M$ for some $M>0$, we infer
\begin{equation*}
	E_1(t,0) |w-z|_{{}_{L^2}}\leq 
		C\int_0^t \Omega^\varepsilon(\zeta)\Bigl\{|w-z|_{{}_{L^2}}^2
			+|z|_{{}_{L^2}}^2+(t-s)^{-1/2}\,e^{-C(t-s)}\Bigr\}E_1(s,0)\,ds.
\end{equation*}
Let us set
\begin{equation*}
	N(t):= \sup_{s\in[0,t]} |w-z|_{{}_{L^2}}\,E_1(s,0)
\end{equation*}
Then, since $|z|_{{}_{L^2}}\leq e^{-2C\,t}E_1(0,t)|w_0|_{{}_{L^2}}$,  we infer
\begin{equation*}
	\begin{aligned}
	E_1(t,0) &|w-z|_{{}_{L^2}}\leq C\int_0^t \Omega^\varepsilon(\zeta) N^2(s)\,E_1(0,s)\,ds\\
		&+C\int_0^t \Omega^\varepsilon(\zeta)\Bigl\{ e^{-4C(t-s)}E_1(0,t)^2|w_0|_{{}_{L^2}}^2
			+(t-s)^{-1/2}\,e^{-C(t-s)}\Bigr\}E_1(s,0)\,ds	\\
		&\quad
		\leq C\int_0^t \Omega^\varepsilon(\zeta) N^2(s)\,E_1(0,s)\,ds
			+C|\Omega^\varepsilon|_{{}_{\infty}}\Bigl(E_1(0,t)|w_0|_{{}_{L^2}}^2+E_1(t,0)\Bigr)
	\end{aligned}
\end{equation*}
since $\lambda_1$ is negative.
By assumption \eqref{meta}, $\lambda_1^\varepsilon\leq -C\Omega^\varepsilon$
for some $C>0$, hence 
\begin{align*}
	\int_0^t \Omega^\varepsilon(\zeta)N^2(s)\,E_1(0,s)\,ds&\leq
      	\int_0^t \Omega^\varepsilon(\zeta)
		N^2(s)\,\exp\left(-C\int_0^s\Omega^\varepsilon(\zeta)\,d\sigma\right)\,ds\\ 
	&\leq N^2(t) \biggl\{ 1- \exp\left(-C\int_0^t \Omega^\varepsilon(\zeta)\,d\sigma\right)\biggr\}.
\end{align*}
so that we obtain the inequality
\begin{equation*}
       \begin{aligned}
       	E_1(t,0)|w-z|_{{}_{L^2}} &\leq CN^2(t) 
       	\biggl\{1- \exp\left(-C\int_0^t \Omega^\varepsilon(\zeta)\,d\sigma\right) \biggr\}\\
		&\quad
		+C|\Omega^\varepsilon|_{{}_{\infty}}\Bigl(E_1(0,t)|w_0|_{{}_{L^2}}^2+E_1(t,0)\Bigr)
       \end{aligned}
\end{equation*}
Taking the supremum, we end up with the estimate 
\begin{equation*}
	N(t) \leq A N^2(t)+B
		\qquad\textrm{with}\quad
		\left\{\begin{aligned}
		A&:=C\left\{1- \exp\left(-C\int_0^t \Omega^\varepsilon(\zeta)\,d\sigma\right)\right\},\\
		B&:=C|\Omega^\varepsilon|_{{}_{\infty}}\Bigl(E_1(0,t)|w_0|_{{}_{L^2}}^2+E_1(t,0)\Bigr)
		\end{aligned}\right.
\end{equation*}
Hence, as soon as
\begin{equation}\label{timecond}
  4AB=4C^2|\Omega^\varepsilon|_{{}_{\infty}}\Bigl(E_1(0,t)|w_0|_{{}_{L^2}}^2+E_1(t,0)\Bigr)
	\Bigl( 1- \exp\left(-C\int_0^t \Omega^\varepsilon(\zeta)\,d\sigma\right)\Bigr)<1
\end{equation}
there holds
\begin{equation*}
	N(t)\leq  \frac{2B}{1+\sqrt{4AB}}\leq 2B 
		= C|\Omega^\varepsilon|_{{}_{\infty}}\Bigl(E_1(0,t)|w_0|_{{}_{L^2}}^2+E_1(t,0)\Bigr)
\end{equation*}
that means, in term of the difference $w-z$, 
\begin{equation*}
	|w-z|_{{}_{L^2}}\,\leq C|\Omega^\varepsilon|_{{}_{\infty}}
		\Bigl(E_1(0,t)^2|w_0|_{{}_{L^2}}^2+1\Bigr)
\end{equation*}
Condition \eqref{timecond} gives a constraint on the final time $T^\varepsilon$.
Since $ 1- e^{-C\int_0^t \Omega^\varepsilon(\zeta)\,d\sigma}\leq 1$ and $E_1(0,t)\leq 1$,
it is enough to require
\begin{equation*}
         4C^2\,|\Omega^\varepsilon|_{{}_{\infty}}
 	\Bigl(|w_0|^2_{{}_{L^2}}+E_1(t,0) \Bigr) <1
\end{equation*}
to assure condition \eqref{timecond} is satisfied.
The latter constraint can be rewritten as
\begin{equation*}
	C\,\exp\left(-\int_0^t \Omega^\varepsilon(\zeta)\,d\sigma\right)
	\leq \exp\left(-\int_0^t \lambda^\varepsilon_1(\zeta)\,d\sigma\right)
	=E_1(t,0) \leq \frac{C}{|\Omega^\varepsilon|_{{}_{\infty}}} -|w_0|^2_{{}_{L^2}},
\end{equation*}
so that we can choose $T^\varepsilon$ of the form
\begin{equation*}
 	T^\varepsilon:=\frac{1}{|\Omega^\varepsilon|_{{}_{\infty}}}
		\ln\left(\frac{C}{|\Omega^\varepsilon|_{{}_{\infty}}} -|w_0|^2_{{}_{L^2}}\right)
		\sim -C\,|\Omega^\varepsilon|_{{}_{\infty}}^{-1}
		\ln |\Omega^\varepsilon|_{{}_{\infty}}
\end{equation*}
for $w_0$ sufficiently small.
\end{proof}

As a consequence  of the estimate \eqref{boundw},
for $|w_0|_{{}_{L^2}}< M$ for some $M>0$, the function $\zeta$ satisfies
\begin{equation}\label{eqzeta}
	\frac{d\zeta}{dt}=\theta^\varepsilon(\zeta)\bigl(1 + r\bigr)
	\qquad\textrm{with}\quad
	|r|\leq C\bigl(|w_0|_{{}_{L^2}}\,e^{-C\,t}+|\Omega^\varepsilon|_{{}_{\infty}}\bigr).
\end{equation}
where the constant $C$ depends also on $M$.
In particular, if $\varepsilon$ and $|w_0|_{{}_{L^2}}$ are small, the function $\zeta=\zeta(t)$
has similar decay properties of the function $\eta$, solution to the reduced Cauchy problem
\begin{equation*}
	\frac{d\eta}{dt}=\theta^\varepsilon(\eta),\qquad \eta(0)=\zeta_0.
\end{equation*}
This preludes to the following consequence of Theorem \ref{thm:metaL}.

\begin{corollary}\label{cor:metaL2}
Let hypotheses {\bf H1-2-3} and \eqref{meta} be satisfied. Assume also
\begin{equation}
	s\,\theta^\varepsilon(s)<0\quad\textrm{ for any } s\in I,\,s\neq 0
	\qquad\textrm{ and }\qquad
	{\theta^\varepsilon}'(\bar \zeta)<0.
\end{equation}
Then, for $\varepsilon$ and $|w_0|_{{}_{L^2}}$ sufficiently small, the estimate \eqref{boundw} holds 
globally in time and the solution $(\zeta,w)$ converges exponentially fast to $(\bar \zeta,0)$ as $t\to+\infty$.
\end{corollary}

\begin{proof}
Thanks to assumption {\bf H1}, for $\varepsilon$ and $|w_0|_{{}_{L^2}}$ sufficiently small,
estimate \eqref{boundw} holds.
Hence, for any initial datum $\zeta_0$, the variable $\zeta=\zeta(t)$ satisfies \eqref{eqzeta}.
and, as a consequence, it converges exponentially fast to $\bar \zeta$ as $t\to+\infty$,
i.e. there exists $\beta^\varepsilon>0$ such that 
$|\zeta-\bar \zeta|\leq |\zeta_0|e^{-\beta^\varepsilon t}$
for any $t$ under consideration.

Furthermore, from \eqref{eqwk_bis}, we deduce
\begin{equation*}
	w_k(t)=w_k(0)\,\exp\left(\int_{0}^t \lambda^\varepsilon_k\,d\sigma\right)
		+\int_{0}^{t} \langle \psi^\varepsilon_k,F\rangle(s)
		\,\exp\left(\int_{s}^t \lambda^\varepsilon_k\,d\sigma\right)\,ds
\end{equation*}
Setting $\Lambda^\varepsilon_1:=\sup\{\lambda^\varepsilon_1(\zeta)\,:\,\zeta\in J\}$,
by the Jensen's inequality, we infer the estimate
\begin{equation*}
	\begin{aligned}
	|w|_{{}_{L^2}}^2(t)&\leq C\left\{|w_0|_{{}_{L^2}}^2\,e^{2\Lambda^\varepsilon_1\,t}
		+\sum_{k}\left(\int_{0}^{t} \langle \psi^\varepsilon_k,F\rangle(s)
		\,e^{\Lambda^\varepsilon_1(t-s)}\,ds\right)^2\right\}\\
		&\leq C\left\{|w_0|_{{}_{L^2}}^2\,e^{2\Lambda^\varepsilon_1\,t}
		+t\,\int_{0}^{t} |F|_{{}_{L^2}}^2(s)\,e^{2\Lambda^\varepsilon_1(t-s)}\,ds\right\}
	\end{aligned}
\end{equation*}
Let $\nu^\varepsilon>0$ be such that $|F|_{{}_{L^2}}(t)\leq C\,e^{-\nu^\varepsilon\,t}$; 
then, if $\nu^\varepsilon \neq |\Lambda_1^\varepsilon|$, there holds
\begin{equation*}
	 |w|_{{}_{L^2}}^2(t)\leq C\left\{|w_0|_{{}_{L^2}}^2\,e^{2\Lambda^\varepsilon_1\,t}
		+t\left(e^{-2\nu^\varepsilon \,t}+e^{2\Lambda^\varepsilon_1\,t}\right)\right\}
\end{equation*}
showing the exponential convergence to $0$ of the component $w$.
\end{proof}

Let us also stress that in the regime $(\zeta,w)\sim (\bar \zeta,0)$, a linearization at the equilibrium solution
$U^\varepsilon(x;\bar \zeta)$ would furnish a more detailed description of the dynamics, since the
source term due to the approximation at an approximate steady state would not be present.
In fact, the description given by the quasi-linearization is meaningful in the regime far from
equilibrium and its aim is to describe the slow motion around a manifold of approximate
solutions.

\section{Application to scalar viscous conservation laws}\label{Sect:Appl}

Next, our aim is to show how the general approach just presented
applies to the case of scalar conservation laws with viscosity.
Specifically, given $\ell>0$, we consider the nonlinear equation
\begin{equation}\label{cauchyBurg1}
		\partial_t u +\partial_x f(u)	=\varepsilon\,\partial_x^2u 
		\qquad\qquad x\in I:=(-\ell,\ell)
\end{equation}
with initial and boundary conditions given by
\begin{equation}\label{cauchyBurg2}
	u(x,0)=u_0(x)\qquad x\in I,
	\qquad\textrm{and}\qquad
	u(\pm\ell,t)=u_\pm \qquad t>0.
\end{equation}
for some $\varepsilon>0$, $u_\pm\in\mathbb{R}$.
We assume that the flux $f$ and the data $u_\pm$ satisfy the conditions 
\begin{equation}\label{fluxhyp2}
	f''(u)\geq c_0>0,\qquad f'(u_+)<0<f'(u_-),\qquad f(u_+)=f(u_-).
\end{equation}
The single value $u\in (u_+,u_-)$ such that $f'(u)=0$ is denoted by $u_\ast$.
Without loss of generality, we assume $f(u_\ast)=0$.

To clarify the relevance of the requirements \eqref{fluxhyp2} and to justify the subsequent choice
for the manifold $\{U^{\varepsilon}(\cdot;\xi)\,:\,\xi\in J\}$, we propose a
digression on the dynamics determined by the problem \eqref{cauchyBurg1}-\eqref{cauchyBurg2}
in the vanishing viscosity limit.

\subsection*{The hyperbolic dynamics}
Setting $\varepsilon=0$, equation \eqref{cauchyBurg1} reduces to the first-order equation
of hyperbolic type
\begin{equation}\label{unvisc}
	\partial_t u + \partial_x f(u)=0
\end{equation}
to be considered together with \eqref{cauchyBurg2}.
The boundary conditions are understood in the sense of Bardos--leRoux--N\'ed\'elec 
\cite{BardLeRoNede79}, meaning that the trace of the solution at the boundary is
requested to take values in appropriate sets.
To be precise, let $u_\ast\in(u_+,u_-)$ be such that $f'(u_\ast)=0$ and set
 \begin{equation*}
  \mathcal{R}u:=\left\{\begin{aligned}
   	&w		&\qquad	&\textrm{if}\;\exists w\neq u\textrm{ s.t. }f(w)=f(u),\\
   	&u_\ast	&\qquad	&\textrm{if }u=u_\ast,
                  \end{aligned}\right.
 \end{equation*}
Then, skipping the details (see \cite{MascTerr99}), the conditions $u(\pm\ell,t)=u_\pm$ translate into
\begin{equation*}
	u(-\ell+0,t)\in  (-\infty,\mathcal{R}u_-)]\cup\{u_-\},\qquad
	u(\ell-0,t)\in  \{u_+\}\cup[\mathcal{R}u_+,+\infty)
\end{equation*} 
Since $f(u_+)=f(u_-)$, there holds $\mathcal{R}u_\pm=u_\mp$, and
the conditions can be rewritten as 
\begin{equation*}
	u(-\ell+0,t)\in  (-\infty,u_+]\cup\{u_-\},\qquad
	u(\ell-0,t)\in  \{u_+\}\cup[u_-,+\infty)
\end{equation*} 
From the boundary conditions, it follows that characteristic curves entering in the domain
from the left side $x=-\ell$, respectively, from the right $x=\ell$, possess speed $f'(u_-)$,
resp. speed $f'(u_+)$.

For \eqref{unvisc} with conditions \eqref{cauchyBurg2} a {\it finite-time stabilization phenomenon}
holds, similar to the one  showed for the first time in \cite{Liu78} in the case of the Cauchy  problem.

\begin{theorem}\label{thm:stabhyp}
Let $u_+<0<u_-$ and $f$ be such that \eqref{fluxhyp2} holds.
Then, for any $u_0\in \textrm{BV}(-\ell,\ell)$, the solution $u$ to the initial-boundary
value problem \eqref{unvisc}--\eqref{cauchyBurg2} is such that
for some $T>0$  and $\xi\in[-\ell,\ell]$, there holds
\begin{equation*}
	u(x,T)=U_{{}_{\textrm{hyp}}}(x;\xi)
		:= u_-\chi_{{}_{(-\ell,\xi)}}(x) +u_+\chi_{{}_{(\xi,\ell)}}(x)
\end{equation*}
for almost any $x$ in $I$.
\end{theorem}

The proof of the statement relies on the {\it theory of generalized characteristics}, 
introduced  in \cite{Dafe77}.
The convexity assumption on the flux function $f$ guarantees that for any point 
$(x,t)\in (-\ell,\ell)\times (0,+\infty)$ there exist a minimal, respectively maximal, 
backward characteristics, which are classical characteristic curves, hence
straight lines with slope $f'(u(x-0,t))$, resp. $f'(u(x+0,t))$. 

By means of such technique it is possible to follow the evolution of the curves
\begin{equation*}
	\begin{aligned}
		\zeta_-(t)&:=\sup\{x\in I\,:\,u(y,t)=u_-\quad\forall\,y\in(-\ell,x)\}\cup\{-\ell\},\\
		\zeta_+(t)&:=\inf\{x\in I\,:\,u(y,t)=u_+\quad\forall\,y\in(x,\ell)\}\cup\{\ell\}.
	\end{aligned}
\end{equation*}
As an illustrative example, let us first consider the case of a non-increasing
initial datum $u_0$.
Then, for any $t>0$, $u(\cdot, t)$ is non-increasing.
If $\zeta_\pm$ are classical characteristics, the difference between their speeds
of propagation satisfies
\begin{equation*}
	\begin{aligned}
	\frac{d\zeta_+}{dt}-\frac{d\zeta_-}{dt}
		&=f'(u_+)-f'(u_-)\\
		&\leq \frac{f(u)-f(u_+)}{u-u_+}-\frac{f(u_-)-f(u)}{u_--u}
		=\frac{f(u_\pm)-f(u)}{(u_--u)(u-u_+)}\ldbrack u\rdbrack=:-\Phi(u),
	\end{aligned}
\end{equation*}
for any $u\in(u_+,u_-)$.
Since $A:=\inf\{\Phi(u)\,:\,u\in(u_+,u_-)\}$ is strictly positive, the two curves intersect
at a time $T$ that is smaller than $2\ell/A$.

The complete rigorous proof of Theorem \ref{thm:stabhyp} requires more technicalities
and it is reported here for completeness.

\begin{proof}
Let $u=u(x,t)$ be the solution to the initial-boundary value problem under consideration
with initial datum $u_0$.
For later use, we set
In particular, $\zeta_-\leq \zeta_+$. 
We are going to show that $\zeta_-(T)=\zeta_+(T)$ for some $T>0$.
\vskip.15cm

{\it 1. There exists $T_0>0$ such that $u(x,t)\in[u_+,u_-]$ for any $x\in(-\ell,\ell)$.}
\vskip.15cm

Indeed, let $\overline u$ be the solution to the Riemann problem for \eqref{unvisc} with datum
\begin{equation*}
	\bar u_0(x)=\left\{\begin{aligned}
		&u_-					&\qquad 	&x<-\ell,\\
		&\max\{u_-,\sup u_0\}	&\qquad 	&x>-\ell,\\
	\end{aligned}\right.
\end{equation*}
Hence, the restriction of $\bar u$ to $(-\ell,\ell)\times(0,\infty)$ is a super-solution to the
initial boundary value problem under consideration and, by comparison principle for
entropy solution, we infer $u(x,t)\leq \bar u(x,t)$. 
Since $\bar u(x,t)=u_-$ for any $x<f'(u_-)\,t-\ell$, there holds
\begin{equation*}
	u(x,t)\leq u_-\qquad\quad \textrm{for}\; x\in(-\ell,\ell),\; t\geq 2\ell/f'(u_-).
\end{equation*}
A similar estimate from below can be obtained by considering as subsolution
the restriction of $\underline u$ to $(-\ell,\ell)\times(0,\infty)$, where $\underline u$
is the solution to \eqref{unvisc} with initial datum
\begin{equation*}
	\bar u_0(x)=\left\{\begin{aligned}
		&\min\{u_+,\inf u_0\}		&\qquad 	&x<\ell,\\
		&u_+				&\qquad 	&x>\ell,\\
	\end{aligned}\right.
\end{equation*}
From now on, we assume 
that the solution $u$ takes values in the interval $[u_-,u_+]$.
\vskip.15cm

{\it 2. Assume that $-\ell<\zeta_-(t)\leq \zeta_+(t)<\ell$ for any $t$; 
then there exists $T_1>0$ such that $u(\zeta_-(t)+0,t)<u_-$ and 
$u_+<u(\zeta_+(t)-0,t)$ for any $t>T_1$.}
\vskip.15cm

If $u$ is continuous at $(\zeta_-(\tau),\tau)$ for some $\tau>0$, then $u(\zeta_-(\tau)+0,t)=u_-$.
Therefore, the maximal backward characteristic from $(\zeta_-(\tau),\tau)$ is the straight line
$x=\zeta_-(\tau)+f'(u_-)(t-\tau)$.
For $\tau>2L/f'(u_-)$, such curve intersects the boundary $x=-\ell$ at some $\sigma\in(0,\tau)$. 
By continuity, all of the maximal backward characteristics from $(\xi,\tau)$ with $\xi>\zeta_-(t)$
and sufficiently close to $\zeta_-(\tau)$ intersect the boundary $x=-\ell$ at some time $\sigma_\ast(\xi)$
smaller than $\sigma$ and close to it.
Because of the boundary conditions, this may happen if and only if $u(\xi,\tau)=u_-$.
Hence, $u(x,\tau)=u_-$ for $x\in(\zeta_-(\tau),\zeta_-(\tau)+\varepsilon)$ for some $\varepsilon>0$,
in contradiction with the definition of $\zeta_-$.
Thus, continuity of $u$ at $(\zeta_-(\tau),\tau)$  may happen only for $\tau\leq 2L/f'(u_-)$.
A similar assertion holds for $\zeta_+$.
\vskip.15cm

{\it 3.  There exist $T>0$  and $\xi\in[-\ell,\ell]$ such that 
$u(x,t)=U_{{}_{\textrm{hyp}}}(\cdot;\xi)$ for any $t\geq T$.}
\vskip.15cm

Given $\theta>0$, let $T_\theta:=2\ell/\theta$ be such that
\begin{equation*}
	u_-^\theta:=u(\zeta_-(T_\theta)+0,T_\theta)<u_-
		\qquad\textrm{and}\qquad 
	u_+<u_-^\theta:=u(\zeta_+(T_\theta)-0,T_\theta).
\end{equation*} 
Let $x_{-}^{\theta}$ be the maximal backward characteristic from $(\zeta_-(T_\theta),T_\theta)$,
whose equation is $x=\zeta_-(T_\theta)+f'(u_-^\theta)(t-T_\theta)$.
If $x_{-}^{\theta}$ hits the right boundary $x=\ell$ at some positive time, the solution $u$
coincides with $U_{{}_{\textrm{hyp}}}(x;\zeta_-(T_\theta))$.
Otherwise, there holds $\zeta_-(T_\theta)-f'(u_-^\theta)T_\theta<\ell$, which gives
\begin{equation*}
	f'(u_-^\theta)>\frac{\zeta_-(T_\theta)-\ell}{T_\theta}\geq -\frac{2\ell}{T_\theta}=-\theta
\end{equation*}
Similarly, let $x_{+}^{\theta}$ be the maximal backward characteristic from $(\zeta_+(T_\theta),T_\theta)$,
whose equation is $x=\zeta_+(T_\theta)+f'(u_+^\theta)(t-T_\theta)$.
If $x_{+}^{\theta}$ does not intersect the left boundary $x=-\ell$ at some positive time, 
there holds $f'(u_+^\theta)<\theta$.

Hence, for any $\varepsilon>0$, we can choose $\theta$ sufficiently large so that
$u_-^\theta>u_\ast-\varepsilon$ and $u_+^\theta<u_\ast+\varepsilon$.
Thus, we have
\begin{equation*}
	\frac{d\zeta_+}{dt}-\frac{d\zeta_-}{dt}
	<\frac{f(u_+)-f(u_\ast+\varepsilon)}{u_+-u_\ast-\varepsilon}
		-\frac{f(u_-)-f(u_\ast-\varepsilon)}{u_--u_\ast+\varepsilon}
\end{equation*}
which is uniformly negative for $\varepsilon$ sufficiently small.
Hence, the curves $\zeta_+$ and $\zeta_-$ intersect at some finite positive time $T>0$.
\end{proof}

\subsection*{Adding viscosity}
As soon as the viscosity term is switched on, i.e. for $\varepsilon>0$, the number of steady
states for \eqref{cauchyBurg1}--\eqref{cauchyBurg2} drastically reduces with respect to the 
corresponding hyperbolic case. 
Indeed, stationary solution to the problem are implicitly determined by the relation
\begin{equation*}
	\int_{u(x)}^{u_-} \frac{ds}{\kappa-f(s)}=\frac{\ell+x}{\varepsilon}
\end{equation*}
where $\kappa\in(f(u_\pm),+\infty)$ is such that
\begin{equation*}
	\Phi(\kappa):=\int_{u_+}^{u_-} \frac{ds}{\kappa-f(s)}=\frac{2\ell}{\varepsilon}
\end{equation*}
Assumptions \ref{fluxhyp2} on the flux $f$ imply that $\Phi$ is strictly
decreasing and such that 
\begin{equation*}
	\lim_{\kappa\to f(u_\pm)^+} \Phi(\kappa)=+\infty,\qquad
	\lim_{\kappa\to +\infty} \Phi(\kappa)=0.
\end{equation*}
Therefore, for any $\ell>0$, there exists a unique steady state
for \eqref{cauchyBurg1}--\eqref{cauchyBurg2}.

\begin{example}\rm
In the case of Burgers equation, $f(u)=u^2/2$, the value $u_+$ coincides with $-u_-$
and $\Phi$ has the explicit form $\sqrt{2}\,\tanh^{-1}(u_-/\sqrt{2\kappa})/\sqrt{\kappa}$,
so that the value $\sigma$ determining the stationary solution is
uniquely determined by the relation
\begin{equation*}
	\sqrt{2\kappa}\,\tanh(\sqrt{2\kappa}\,\ell/\varepsilon)=u_-.
\end{equation*}
Given $\kappa$, the steady state $U$ has the expression
$U(x)=\sqrt{2\kappa}\,\tanh(-\sqrt{2\kappa}\,x/\varepsilon)$.
\end{example}

Following the general approach introduced in the previous section, we build a
one-parameter family of functions $U^\varepsilon=U^\varepsilon(\cdot;\xi)$ with
$\xi\in J$ converging to $U_{{}_{\textrm{hyp}}}(\cdot;\xi)$ as $\varepsilon\to 0$.
In particular, the parameter set $J$ coincides with the interval $I$
\begin{equation*}
	u(x,T)=U_{{}_{\textrm{hyp}}}(x;\xi)
		:= u_-\chi_{{}_{(-\ell,\xi)}}(x) +u_+\chi_{{}_{(\xi,\ell)}}(x)
\end{equation*}
There are many meaningful choices for $U^\varepsilon$
(see the traveling wave approach in \cite{deGrKara98});
here, we opt for matching at a given point $\xi\in I$
the two stationary solutions  of \eqref{cauchyBurg1} in $(-\ell,\xi)$ and $(\xi,\ell)$,
denoted by $U^\varepsilon_-$ and $U^\varepsilon_+$,
satisfying the boundary conditions
\begin{equation*}
	U^\varepsilon_-(-\ell;\xi)=u_-,\; U^\varepsilon_-(\xi;\xi)=u_\ast
	\qquad\textrm{and}\qquad
	U^\varepsilon_+(\xi;\xi)=u_\ast,\; U^\varepsilon_+(\ell;\xi)=u_+
\end{equation*}
where $u_\ast$ is such that $f'(u_\ast)=0$.
Hence, we set
\begin{equation*}
	U^{\varepsilon}(x;\xi)=\left\{\begin{aligned}
		&U^\varepsilon_-(x;\xi) 	&\qquad &	-\ell<x<\xi<\ell \\
		&U^\varepsilon_+(x;\xi)	&\qquad &-\ell<\xi<x<\ell ,
           \end{aligned}\right.
\end{equation*}
Given $\kappa\in(f(u_\pm),+\infty)$ and $u\in(u_+,u_-)$, let us define
\begin{equation*}
	\Psi_\ast(\kappa, u)=\int_{u_\ast}^{u} \frac{ds}{\kappa-f(s)}
\end{equation*}
Similarly to the case of stationary states, the function $\Phi$ is such that
\begin{equation*}
	\begin{aligned}
	&\Psi_\ast(\kappa_-,\cdot)\;\textrm{decreasing},\qquad
		&\Psi_\ast(\kappa_-,f(u_\pm))=+\infty,\qquad 
		&\Psi_\ast(\kappa_-,+\infty)=0,\\
	&\Psi_\ast(\kappa_+,\cdot)\;\textrm{increasing}\qquad
		&\Psi_\ast(\kappa_+,f(u_\pm))=-\infty,\qquad
		&\Psi_\ast(\kappa_+,+\infty)=0,
	\end{aligned}
\end{equation*}
so that for any $\xi\in(-\ell,\ell)$ there are (unique)
$\kappa_\pm^\varepsilon=\kappa_\pm^\varepsilon(\xi)\in(f(u_\pm),+\infty)$ such that
\begin{equation}\label{defkappapm}
	\varepsilon\Psi_\ast(\kappa_\pm^\varepsilon,u_\pm)\pm\ell=\xi
\end{equation}
Correspondingly, functions $U^\varepsilon_\pm$ are implicitly given by
\begin{equation*}
	\varepsilon\Psi_\ast(\kappa_\pm^\varepsilon,U^\varepsilon_\pm(x;\xi))+x=\xi.
\end{equation*}
By substitution, denoting by $\delta_{x=\xi}$ the Dirac's delta distribution concentrated
at $x=\xi$, there holds in the sense of distributions
\begin{equation}\label{FUeps}
	{\mathcal F}^\varepsilon[U^{\varepsilon}(\cdot;\xi)]
		=\ldbrack \partial_x U^{\varepsilon}\rdbrack_{{}_{x=\xi}}\delta_{{}_{x=\xi}}
		=\frac{1}{\varepsilon}\bigl(\kappa_-^\varepsilon(\xi)-\kappa_+^\varepsilon(\xi)\bigr)\delta_{{}_{x=\xi}}
\end{equation}
with $\kappa^\varepsilon_\pm$ implicitly defined by \eqref{defkappapm}.
As a consequence of the properties of function $\Phi$, the difference function
$\xi\mapsto \kappa_-^\varepsilon(\xi)-\kappa_+^\varepsilon(\xi)$ is monotone 
decreasing and such that
\begin{equation*}
	\lim_{\xi\to \pm\ell^\mp} \bigl(\kappa_-^\varepsilon(\xi)-\kappa_+^\varepsilon(\xi)\bigr)
		=\mp\infty.
\end{equation*}
Then, there exists unique $\xi_\ast\in(-\ell,\ell)$ such that
$(\kappa_-^\varepsilon-\kappa_+^\varepsilon)(\xi_\ast)=0$
and such a value is such that $U^\varepsilon(\cdot;\xi_\ast)$ is 
the unique steady state of the problem.

From the bounds
\begin{equation*}
	\begin{aligned}
		f(u_\pm)+f'(u_+)(u-u_+) &\leq
			f(u)\leq \frac{f(u_\pm)}{u_\ast-u_+}(u_\ast-u)	&\qquad	&u\in[u_+,u_\ast],\\
		f(u_\pm)-f'(u_-)(u_--u) &\leq
			f(u)\leq \frac{f(u_\pm)}{u_--u_\ast}(u-u_\ast)	&\qquad	&u\in[u_\ast,u_-],
	\end{aligned}
\end{equation*}
we locate approximately the differences $\kappa_\pm^\varepsilon(\xi)-f(u_\pm)$
\begin{equation*}
	\begin{aligned}
	\frac{-f'(u_+)(u_\ast-u_+)}{\exp\{-f'(u_+)(\ell-\xi)/\varepsilon\}-1}
		&\leq\kappa_+^\varepsilon(\xi)-f(u_\pm)
		\leq \frac{f(u_\pm)}{\exp\{f(u_\pm)(\ell-\xi)/\varepsilon(u_\ast-u_+)\}-1}\\
	 \frac{f'(u_-)(u_--u_\ast)}{\exp\{f'(u_-)(\ell+\xi)/\varepsilon\}-1}
	 	&\leq\kappa_-^\varepsilon(\xi)-f(u_\pm)
		\leq 
	\frac{f(u_\pm)}{\exp\{f(u_\pm)(\ell+\xi)/\varepsilon(u_--u_\ast)\}-1}.
	\end{aligned}
\end{equation*}
Such bounds show that $|\kappa_-^\varepsilon-\kappa_+^\varepsilon|$ is exponentially
small as $\varepsilon\to 0^+$, uniformly in any compact subset of $(-\ell,\ell)$;
therefore, for any $\delta\in(0,\ell)$, there exist $C_1, C_2>0$, indipendent on $\varepsilon$,
such that
\begin{equation}\label{boundUx}
	\bigl|\ldbrack \partial_x U^{\varepsilon}\rdbrack_{{}_{x=\xi}}\bigr|
	\leq C_1\,e^{-C_2/\varepsilon}
	\qquad\qquad \forall\,\xi\in(-\ell+\delta,\ell-\delta).
\end{equation}
In particular, hypothesis {\bf H1}, stated in Section 2, is satisfied.

Going further, retracing the definitions previously introduced and setting 
$a^\varepsilon:=f'(U^{\varepsilon})$, we consider the operators
\begin{equation*}
	\mathcal{L}^\varepsilon_\xi v:=\varepsilon v''-\bigl(a^{\varepsilon}(\cdot;\xi)\,v\bigr)'
			\qquad\qquad
	\mathcal{L}^{\varepsilon,\ast}_\xi v:=\varepsilon v''+a^{\varepsilon}(\cdot;\xi)\,v'
\end{equation*}
where the adjoint operator $\mathcal{L}^{\varepsilon,\ast}_\xi$ is considered with Dirichlet 
boundary conditions.

For small $\varepsilon$ and $v$, the dynamics of the parameter $\xi$ is
approximately given by 
\begin{equation*}
	\frac{d\xi}{dt}\approx \theta^\varepsilon(\xi), 
	\qquad\textrm{where}\quad
 	\theta^\varepsilon(\xi):=\langle \psi^\varepsilon_1,{\mathcal F[U^{\varepsilon}] \rangle}\\
\end{equation*}
where $\psi^\varepsilon_1$ is the first eigenfunction of the adjoint operator
$\mathcal{L}^{\varepsilon,\ast}_\xi$ satisfying the normalization condition
\begin{equation}\label{normalization}
	\langle \psi^\varepsilon_1(\cdot;\xi), \partial_{\xi}U^{\varepsilon}(\cdot;\xi)\rangle=1,
\end{equation}
For $\varepsilon\sim 0$, the eigenfunction $\psi_1^\varepsilon$ is close to the eigenfunction
of  ${\mathcal L}^{0,\ast}_\xi$ relative to the eigenvalue $\lambda=0$, with
\begin{equation*}
	a^{0}(x;\xi):=f'(u_-)\chi_{{}_{(-\ell,\xi)}}(x)+f'(u_+)\chi_{{}_{(\xi,\ell)}}(x)
\end{equation*}
Hence, we obtain the representation formula
\begin{equation}\label{firstadjoint}
	\psi_1^\varepsilon(x)\approx C\,\psi_1^0(x)
\end{equation}
where
\begin{equation*}
	\psi_1^0(x):=\left\{\begin{aligned}
			&(1-e^{u_+(\ell-\xi)/\varepsilon})(1-e^{-u_-(\ell+x)/\varepsilon})
					&\qquad	&x<\xi,\\
			&(1-e^{-u_-(\ell+\xi)/\varepsilon})(1-e^{u_+(\ell-x)/\varepsilon})
					&\qquad	&x>\xi,		
		\end{aligned}\right.
\end{equation*}
for some $C\in\mathbb{R}$.
In the limit $\varepsilon\to 0$, we obtain
$\psi_1^\varepsilon\approx C$, provided $\xi$ is bounded away from the boundaries $\pm\ell$.
With the approximation
\begin{equation*}
	U^{\varepsilon}(x;\xi)\approx U_{{}_{\textrm{hyp}}}(x;\xi)
		:= u_-\chi_{{}_{(-\ell,\xi)}}(x) +u_+\chi_{{}_{(\xi,\ell)}}(x)
\end{equation*}	
we infer
\begin{equation*}
	\frac{U^{\varepsilon}(x;\xi+h)-U^{\varepsilon}(x;\xi)}{h}
		\approx -\frac{1}{h}\,\ldbrack u\rdbrack\,\chi_{{}_{(\xi,\xi+h)}}(x)
\end{equation*}
so that we expect $\partial_\xi U^\varepsilon$ to converge to $-\ldbrack u\rdbrack\,\delta_{\xi}$ as 
$\varepsilon\to 0$ in the sense of distributions.
Hence, the normalization condition \eqref{normalization} gives the choice $C=-1/\ldbrack u\rdbrack$
in \eqref{firstadjoint}.
Therefore, we deduce an approximate expression for the function $\theta^\varepsilon$
\begin{equation*}
 	\theta^\varepsilon(\xi)\approx -\frac{1}{\ldbrack u\rdbrack}
		\,\langle 1,{\mathcal F[U^{\varepsilon}] \rangle}
		=\frac{1}{\varepsilon\,\ldbrack u\rdbrack}
		\left(\kappa_+^\varepsilon(\xi)-\kappa_-^\varepsilon(\xi)\right).
\end{equation*}
Estimate \eqref{boundUx} shows that the the function $\theta$
has order of magnitude $e^{-C/\varepsilon}$.

\begin{example}\rm
In the very special case $f(u)=|u|$, with $u_\ast=0$ and $u_+=-u_-$, the earlier
estimates on $\kappa_\pm^\varepsilon$ are exact, so that
\begin{equation*}
	\frac{\kappa_+^\varepsilon(\xi)}{u_-}=1+\frac{e^{-(\ell-\xi)/\varepsilon}}{1-e^{-(\ell-\xi)/\varepsilon}}
	\qquad
	\frac{\kappa_-^\varepsilon(\xi)}{u_-}=1+\frac{e^{-(\ell+\xi)/\varepsilon}}{1-e^{-(\ell+\xi)/\varepsilon}}.
\end{equation*}
In this case, the function $\theta^\varepsilon$ is approximated by 
\begin{equation*}
	\theta^\varepsilon(\xi)\approx 
		\frac{1}{2\varepsilon}\left(\frac{e^{-(\ell+\xi)/\varepsilon}}{1-e^{-(\ell+\xi)/\varepsilon}}
		-\frac{e^{-(\ell-\xi)/\varepsilon}}{1-e^{-(\ell-\xi)/\varepsilon}}\right)
\end{equation*}
which gives $\theta^\varepsilon(\xi)\approx -\varepsilon^{-1}\,e^{-\ell/\varepsilon}\sinh(\xi/\varepsilon)$
in the regime $\varepsilon\to 0^+$.
\end{example}
\vskip.15cm

\begin{example}\label{example:BurgersEq} \rm
For the Burgers equation, $f(u)=u^2/2$, there holds
\begin{equation*}
	\Psi_\ast(\kappa, u)=2\int_{u_\ast}^{u} \frac{ds}{2\kappa-s^2}
		=\frac{\sqrt{2}}{\sqrt{\kappa}}\,\tanh^{-1}\left(\frac{u}{\sqrt{2\kappa}}\right)
\end{equation*}
Given $\xi\in(-\ell,\ell)$, the values $\kappa_\pm^\varepsilon$ can be approximated by
$\tilde \kappa_\pm^\varepsilon$ determined by
\begin{equation*}
	\frac{2\varepsilon}{u_-}
		\,\tanh^{-1}\left(\frac{-u_-}{\sqrt{2\tilde \kappa_+^\varepsilon}}\right)+\ell=\xi,\qquad
	\frac{2\varepsilon}{u_-}
		\,\tanh^{-1}\left(\frac{u_-}{\sqrt{2\tilde \kappa_-^\varepsilon}}\right)-\ell=\xi.
\end{equation*}
obtained by substituting the multiplicative term $\sqrt{2}/\sqrt{\kappa_\pm^\varepsilon}$
with $\sqrt{2}/\sqrt{f(u_\pm)}=2/u_-$.
By computation, we obtain the explicit expressions
\begin{equation*}
	\tilde \kappa_+^\varepsilon=\frac{u_-^2}{2}\,
		\frac{1}{\tanh^2\left\{u_-(\ell-\xi)/2\varepsilon\right\}},
		\qquad
	\tilde \kappa_-^\varepsilon=\frac{u_-^2}{2}\,
		\frac{1}{\tanh^2\left\{u_-(\ell+\xi)/2\varepsilon\right\}}.
\end{equation*}
Since, for $x,y>0$, 
\begin{equation*}
	\begin{aligned}
	\frac{1}{\tanh^2(x/\varepsilon)}-\frac{1}{\tanh^2(y/\varepsilon)}
	&=\frac{4\bigl(e^{(y-x)/\varepsilon}-e^{(x-y)/\varepsilon}\bigr)
		\bigl(e^{(x+y)/\varepsilon}-e^{-(x+y)/\varepsilon}\bigr)}
		{(e^{x/\varepsilon}-e^{-x/\varepsilon})^2(e^{y/\varepsilon}-e^{-y/\varepsilon})^2}\\
	&\approx 4\,\bigl(e^{-2x/\varepsilon}-e^{-2y/\varepsilon}\bigr)
	\end{aligned}
\end{equation*}
as $\varepsilon\to 0^+$, the function $\theta^\varepsilon$ approaches
\begin{equation*}
	\theta^\varepsilon(\xi)\approx 
	\frac{1}{2\varepsilon\,u_-}\left(\tilde \kappa_-^\varepsilon(\xi)-\tilde \kappa_+^\varepsilon(\xi)\right)
	\approx
	\frac{1}{\varepsilon}u_-\bigl(e^{-u_-(\ell+\xi)/\varepsilon}-e^{-u_-(\ell-\xi)/\varepsilon}\bigr)
\end{equation*}
which corresponds to the formula determined in \cite{ReynWard95b}.
\end{example}

\section{Spectral analysis for scalar diffusion-transport operators}\label{Sect:Spec}

Our concern in the present section is to estabilish a precise description on the location of 
the eigenvalues of the linearized operator, in order to show that the general procedure
developed in Section \ref{Sect:Abs} is indeed applicable in the case of scalar conservation laws 
with convex flux.

The problem of determining the limiting structure of the spectrum of the type of second order 
differential operators we deal with has been widely considered in the literature.
Among others, let us quote the approach, based on the use of Pr\"ufer transform, used in 
\cite{CarrPego89}, in the context of metastability analysis for the Allen--Cahn equation.
Here, we prefer to follow the strategy implemented in \cite{KreiKrei86}, for the linearization
at the steady state of the Burgers equation.
In what follows, we show that the same kind of eigenvalues distribution holds in a much more
general situation, the main ingredient being the resemblance of the coefficient $a^\varepsilon$
to a step function $a^0$, jumping from a positive to a negative value, as $\varepsilon\to 0^+$.
\vskip.15cm

Fixed $\varepsilon>0$ and linearizing the scalar conservation law \eqref{cauchyBurg1} 
at a given a reference profile $U^{\varepsilon}=U^{\varepsilon}(x)$, satisfying the boundary 
conditions $U^{\varepsilon}(\pm\ell)=u_\pm$, we end up with the differential linear
diffusion-transport operator
\begin{equation}\label{transpop}
	\mathcal{L}^{\varepsilon}_\xi u:=u''-(a^{\varepsilon}(x) u)'
	\qquad u(\pm \ell)=0,
\end{equation}
where $a^\varepsilon=a^{\varepsilon}(x):=f'(U^{\varepsilon}(x))$.
The aim of this Section is to describe the structure of the spectrum $\sigma(\mathcal{L}^{\varepsilon}_\xi)$
of the operator  $\mathcal{L}^{\varepsilon}_\xi$ for $\varepsilon$ sufficiently small.

Given the function $a^\varepsilon$, let us introduce the self-adjoint operator
\begin{equation*}
	\mathcal{M}^{\varepsilon}_\xi v:=\varepsilon^2\,v''-b^{\varepsilon}v
	\qquad  v(\pm \ell)=0,
\end{equation*}
where
\begin{equation}\label{defb}
	b^{\varepsilon}:=\left(\frac12\,a^{\varepsilon}\right)^2+\frac12\,\varepsilon\,\frac{da^{\varepsilon}}{dx}.
\end{equation}
A straightforward calculation shows that if $u$ is an eigenfunction of \eqref{transpop} relative to the 
eigenvalue $\lambda$, then the function $v(x)$ defined by
\begin{equation*}
	v(x)=\exp \left(-\frac{1}{2\varepsilon} \int_{x_0}^x a^{\varepsilon}(y)\,dy \right)u(x)
\end{equation*}
(with $x_0$ arbitrarily chosen) is an eigenfunction of the operator $\mathcal{M}^{\varepsilon}_\xi$ relative 
to the eigenvalue $\mu:=\varepsilon\lambda$.
Since $\mathcal{M}^\varepsilon_\xi$ is self-adjoint, we can state that the spectrum of the operator
$\mathcal{L}^\varepsilon_\xi$ is composed by \underline{real} eigenvalues.
Moreover, if $u$ is an eigenfunction of \eqref{transpop} relative to the first eigenvalue $\lambda_1^\varepsilon$,
integrating in $(-\ell,\ell)$ the relation $\mathcal{L}^\varepsilon_\xi u=\lambda\,u$, we deduce the identity
\begin{equation*}
	0=\int_{-\ell}^{\ell}\bigl(\mathcal{L}^{\varepsilon}_\xi-\lambda_1^\varepsilon\bigr)u\,dx
		=\varepsilon\,\bigl(u'(\ell)-u'(-\ell)\bigr)-\lambda_1^\varepsilon\int_{-\ell}^{\ell} u(x)\,dx
\end{equation*}
Assuming, without loss of generality, $u$ to be strictly positive in $(-\ell,\ell)$ and normalized so that its 
integral in $(-\ell,\ell)$ is equal to 1, we get
\begin{equation*}
	\lambda_1^\varepsilon=\varepsilon\,\bigl(u'(\ell)-u'(-\ell)\bigr)<0
\end{equation*}
Hence, for any choice of the function $a^\varepsilon$, there holds
\begin{equation*}
	\sigma(\mathcal{L}^{\varepsilon}_\xi)\subset (-\infty,0).
\end{equation*}
Our next aim is to show that under appropriate assumption on the behavior of the family of functions 
$a^\varepsilon$ as $\varepsilon\to 0^+$, it is possible to furnish a detailed representation of the 
eigenvalue distributions for small $\varepsilon$.
Specifically, we are interested in coefficients $a^\varepsilon$ behaving, in the limit $\varepsilon\to 0^+$
as a step function of the form
\begin{equation*}
	a^0(x):=\left\{\begin{aligned}
		&a_-	&\qquad &x\in(-\ell,\xi),\\ &a_+	&\qquad &x\in(\xi,\ell),
			\end{aligned}\right.
\end{equation*}
for some $\xi\in(-\ell,\ell)$ and $a_+<0<a_-$.
We will show that, under appropriate assumptions making precise in which sense $a^\varepsilon$
``resemble'' $a^0$ for $\varepsilon$ small,  the first eigenvalue $\lambda_1^\varepsilon$ turns to be 
``very close'' to $0$ for $\varepsilon$ small, and all of the others eigenavalues $\lambda_k^\varepsilon$, 
with $k\geq 2$,  are such that $\varepsilon\lambda_k^\varepsilon=O(1)$ as $\varepsilon\to 0^+$.

\subsection*{Estimate from below for the first eigenvalue}
We estimate the first eigenvalue $\mu^\varepsilon_1$ of the operator $\mathcal{M}^\varepsilon_\xi$
by means of the inequality
\begin{equation*}
		|\mu_1^\varepsilon|\leq \frac{|\mathcal M^\varepsilon_\xi\,\psi|_{{}_{L^2}}}{|\psi|_{{}_{L^2}}}.
\end{equation*}
for smooth test function $\psi$ such that $\psi(\pm \ell)=0$.
Let us consider as test function $\psi^\varepsilon(x):=\psi_0^\varepsilon(x)-K^\varepsilon(x)$, where
\begin{equation*}
	\begin{aligned}
 		\psi_0^\varepsilon(x)&:=\exp\left(\frac{1}{2\varepsilon}\int_{\xi}^{x}a^\varepsilon(y)\,dy\right),\\
		K^\varepsilon(x)&:=\frac{1}{2\ell}\bigl\{\psi^\varepsilon_0(-\ell)(\ell-x)+\psi^\varepsilon_0(\ell)(\ell+x)\bigr\}.
	\end{aligned}
\end{equation*} 
A direct calculation shows that $\mathcal{M}^\varepsilon_\xi \psi:=b^\varepsilon\,K$ and, assuming the family 
$b^\varepsilon$ to be uniformly bounded, we infer
\begin{equation*}
		|\mu^\varepsilon_1|\leq \frac{|b^\varepsilon\,K^\varepsilon|_{{}_{L^2}}}{|\psi^\varepsilon_0-K^\varepsilon|_{{}_{L^2}}}
			\leq C\,\frac{|K^\varepsilon|_{{}_{L^2}}}{|\psi^\varepsilon_0|_{{}_{L^2}}-|K^\varepsilon|_{{}_{L^2}}}
			= \frac{C}{|K^\varepsilon|_{{}_{L^2}}^{-1}|\psi^\varepsilon_0|_{{}_{L^2}}-1}
\end{equation*}
as soon as $|\psi^\varepsilon_0|_{{}_{L^2}}>|K^\varepsilon|_{{}_{L^2}}$.

The opposite case being similar, let us assume $\psi_0(-\ell)\geq \psi_0(\ell)$. 
From the definition of $K^\varepsilon$, it follows
\begin{equation*}
	|K^\varepsilon|_{{}_{L^2}}^2
		=\frac{2\ell}{3}\bigl\{\psi_0^2(\ell)+\psi_0(\ell)\psi_0(-\ell)+\psi_0^2(-\ell)\bigr\}
		\leq 2\ell\,\psi_0^2(-\ell).
\end{equation*}
Therefore, we deduce
\begin{equation*}
	|K^\varepsilon|_{{}_{L^2}}^{-2}|\psi^\varepsilon_0|_{{}_{L^2}}^2
		\geq 2\ell\,\psi_0^{-2}(-\ell)\,\int_{-\ell}^{\ell}  |\psi^\varepsilon_0(x)|^2\,dx
		=2\ell\,I^\varepsilon
\end{equation*}
where 
\begin{equation*}
	I^\varepsilon:=\int_{-\ell}^{\ell}  \exp\left(\frac{1}{\varepsilon}\int_{-\ell}^{x}a^\varepsilon(y)\,dy\right)\,dx
\end{equation*}
Since $a^\varepsilon$ converges to the step function $a^0$ as $\varepsilon\to 0^+$,
it is natural to approximate the latter integral in term of the corresponding one for $a^0$:
\begin{equation*}
	I^\varepsilon =\int_{-\ell}^{\ell}  \exp\left(\frac{1}{\varepsilon}\int_{-\ell}^{x}(a^\varepsilon-a^0)(y)\,dy\right)
			\exp\left(\frac{1}{\varepsilon}\int_{-\ell}^{x}a^0(y)\,dy\right)\,dx
	\geq  e^{-|a^\varepsilon-a^0|_{{}_{L^1}}/\varepsilon}\,I^0.
\end{equation*}
Since, for $\varepsilon$ small,
\begin{equation*}
	\begin{aligned}
	I^0&=\int_{-\ell}^{\xi} e^{a_-(x+\ell)/\varepsilon}\,dx
		+e^{a_-(\xi+\ell)/\varepsilon}\int_{\xi}^{\ell} e^{a_+(x-\xi)/\varepsilon}\,dx\\
	&=\varepsilon\,e^{a_-(\xi+\ell)/\varepsilon}\Bigl\{ \frac{1}{a_-}\bigl(1-e^{-a_-(\xi+\ell)/\varepsilon}\bigr)
		-\frac{1}{a_+}\,\bigl(1-e^{a_+(\ell-\xi)/\varepsilon}\bigr)\Bigr\} 
		\sim \frac{[a]}{a_-a_+}\,\varepsilon\,e^{a_-(\xi+\ell)/\varepsilon}.
	\end{aligned}
\end{equation*}
the subsequent estimate holds
\begin{equation*}
	|K^\varepsilon|_{{}_{L^2}}^{-2}|\psi^\varepsilon_0|_{{}_{L^2}}^2
		\geq  2\,\ell\,e^{-|a^\varepsilon-a^0|_{{}_{L^1}}/\varepsilon}\,I^0
		\geq C_1\,e^{C_2/\varepsilon}.
\end{equation*}
whenever $|a^\varepsilon-a^0|_{{}_{L^1}}\leq c_0\varepsilon$ for some $c_0>0$.
Thus, we deduce for the first eigenvalue $\mu^\varepsilon_1$ of the self-adjoint operator $\mathcal{M}^\varepsilon_\xi$
the estimate $|\mu^\varepsilon_1|\leq C_1\,e^{C_2/\varepsilon}$ for some positive constant $C_1, C_2$.
As a consequence,  since the spectrum $\sigma(\mathcal{L}^\varepsilon_\xi)$ coincides with 
$\varepsilon^{-1}\sigma(\mathcal{M}^\varepsilon_\xi)$, the next result holds.

\begin{proposition}\label{prop:firstevalue}
Let $a^\varepsilon$ be a family of functions satisfying the assumption:\\
\indent {\sf A0.} there exists $C>0$, indipendent on $\varepsilon>0$, such that
\begin{equation*}
	|a^\varepsilon|_{{}_{\infty}}+\varepsilon\left| \dfrac{da^\varepsilon}{dx}\right|_{{}_{\infty}}\leq C
\end{equation*}
If there exists $\xi\in(-\ell,\ell)$, $a_+<0<a_-$ and $C>0$ for which $|a^\varepsilon-a^0|_{{}_{L^1}}\leq C\varepsilon$,
then there exist constants $C,c>0$ such that  $-C\,e^{-c/\varepsilon}\leq \lambda^\varepsilon_1<0$.
\end{proposition}

Let us stress that the request  $a_+<0<a_-$ is essential, even if hided in the proof.
If this is not the case, the term $K^\varepsilon$ would not be small as $\varepsilon\to 0^+$
and its $L^2$ norm would not be bounded by the $L^2$-norm of $\psi_0^\varepsilon$.
In fact, the statement in Proposition \ref{prop:firstevalue} may not hold when $a_\pm$ have the
same sign, the easiest example being the case $a^\varepsilon\equiv a_+=a_->0$.

The next Example gives an heuristic estimate for the first eigenvalue $\lambda_1^\varepsilon$.

\begin{example} \rm{
Given $-\alpha<0<\beta$ and $a_\pm\in\mathbb{R}$, let us set $I=(-\alpha,\beta)$, $[a]:=a_+-a_-$ and
\begin{equation*}
	a(x)=a_-\chi_{{}_{(-\alpha,0)}}(x)+a_+\chi_{{}_{(0,\beta)}}(x).
\end{equation*}
Given $\lambda>0$, let us look for functions $u\in C(I)$, such that
\begin{equation*}
	({\mathcal L}-\lambda)u=\varepsilon\,u''-\left(a(x)\,u\right)'-\lambda u=0,\qquad
	u(-\alpha)=u(\beta)=0
\end{equation*}
in the sense of distributions.
Since $a'=[a]\,\delta_{0}$, this amounts in finding two functions
$u^\pm$ such that
\begin{equation*}
	({\mathcal L}_\pm-\lambda) u
		=\varepsilon\,u_\pm''-a_\pm\,u'_\pm+\lambda\,u=0,\qquad
		u_-(-\alpha)=u_+(\beta)=0
\end{equation*}
and the following transmission conditions are satisfied
\begin{equation*}
	u_+(0)-u_-(0)=0\qquad\textrm{and}\qquad
	\varepsilon\,\bigl(u'_+(0)-u'_-(0)\bigr)-[a]\,u_\pm(0)=0.
\end{equation*}
The characteristic polinomial of ${\mathcal L}_\pm$ is 
$p_\pm(\mu;\lambda):=\varepsilon\,\mu^2-a_\pm\,\mu-\lambda$,
with roots
\begin{equation*}
	\mu_-^\pm:=\frac{a_- \pm \Delta_-}{2\varepsilon},
		\qquad
	\mu_+^\pm:=\frac{a_+ \pm \Delta_+}{2\varepsilon},
		\qquad\textrm{where }
	\Delta_\pm:=\sqrt{a_\pm^2+4\,\varepsilon\,\lambda}.
\end{equation*}
Assume $\lambda>-(a_\pm)^2/4\,\varepsilon$.
Choosing $u_\pm$ in the form
\begin{equation*}
	u_-(x)=A_-(e^{\mu_-^+(\alpha+x)}-e^{\mu_-^-(\alpha+x)})
		\quad\textrm{and}\quad
	u_+(x)=A_+(e^{-\mu_+^+(\beta-x)}-e^{-\mu_+^-(\beta-x)}).
\end{equation*}
Setting $\theta_-^\pm:=e^{\mu_-^\pm\alpha}$ and 
$\theta_+^\pm:=e^{-\mu_+^\pm\beta}$, there holds
\begin{equation*}
	\begin{aligned}
	u_-(0)&=A_-(\theta_-^+-\theta_-^-)
		& \qquad u_-'(0)&=A_-(\mu_-^+\theta_-^+-\mu_-^-\theta_-^-)\\
	u_+(0)&=A_+(\theta_+^+-\theta_+^-)
		& \qquad u_+'(0)&=A_+(\mu_+^+\theta_+^+-\mu_+^-\theta_+^-).
	\end{aligned}
\end{equation*}
Therefore, the transmission conditions take the form of a linear system in $A_\pm$
\begin{equation*}
	\left\{\begin{aligned}
		&(\theta_+^+-\theta_+^-)A_+-(\theta_-^+-\theta_-^-)A_-=0,\\
		&\Bigl\{ \left(2\varepsilon\,\mu_+^+-[a]\right)\theta_+^+
			-\left(2\varepsilon\,\mu_+^--[a]\right)\theta_+^-\Bigr\}A_+\\
		&\qquad 
			+\Bigl\{-\left(2\varepsilon\,\mu_-^++[a]\right)\theta_-^+
			+\left(2\varepsilon\,\mu_-^-+[a]\right)\theta_-^-\Bigr\}A_-=0.
	\end{aligned}\right.
\end{equation*}
After some manipulations, the determinant $D=D(\lambda,\varepsilon)$ 
of system can be written as
\begin{equation*}
	D=-\left([a]-[\Delta]\right)\theta_-^+\theta_+^++\left([a]+\{\Delta\}\right)\theta_-^+\theta_+^-
		+\left([a]-\{\Delta\}\right)\theta_-^-\theta_+^+-\left([a]+[\Delta]\right)\theta_-^-\theta_+^-,
\end{equation*}
where $[\Delta]:=\Delta_+-\Delta_-$  and $\{\Delta\}:=\Delta_++\Delta_-$.

Since $\sqrt{\kappa^2+4\,x}=|\kappa|+2|\kappa|^{-1}\,x+o(x)$, 
in the case $a_+<0<a_-$ there hold
\begin{equation*}
	\begin{aligned}
	\{\Delta\}&=\sqrt{a_+^2+4\,\varepsilon\lambda}+\sqrt{a_-^2+4\,\varepsilon\lambda}
		=-[a]\left(1-\frac{2\,\varepsilon\lambda}{a_+\,a_-}\right)+o(\varepsilon\lambda)\\
	[\Delta]&=\sqrt{a_+^2+4\,\varepsilon\,\lambda}-\sqrt{a_-^2+4\,\varepsilon\,\lambda}
		=-\{a\}\left(1+\frac{2\,\varepsilon\lambda}{a_+ a_-}\right)
			+o(\varepsilon\lambda)
	\end{aligned}
\end{equation*}
as $\varepsilon\lambda\to 0$, together with
\begin{equation*}
	\begin{aligned}
	&\varepsilon\ln(\theta_-^+\theta_+^+)
		=\frac{1}{2}\bigl\{(a_-+\Delta_-)\alpha-(a_++\Delta_+)\beta\bigr\}
		= a_-\alpha +\left(\frac{\alpha}{a_-}+\frac{\beta}{a_+}\right)\varepsilon\lambda
			+o(\varepsilon\lambda),\\
	&\varepsilon\ln(\theta_-^+\theta_+^-)
		=\frac{1}{2}\bigl\{(a_-+\Delta_-)\alpha-(a_+-\Delta_+)\beta\bigr\}
		=a_-\alpha-a_+\beta+\left(\frac{\alpha}{a_-}-\frac{\beta}{a_+}\right)\varepsilon\lambda
			+o(\varepsilon\lambda),\\
	&\varepsilon\ln(\theta_-^-\theta_+^+)
		=\frac{1}{2}\bigl\{(a_--\Delta_-)\alpha-(a_++\Delta_+)\beta\bigr\}
		=-\left(\frac{\alpha}{a_-}-\frac{\beta}{a_+}\right)\varepsilon\lambda^\varepsilon
			+o(\varepsilon\lambda),\\
	&\varepsilon\ln(\theta_-^-\theta_+^-)
		=\frac{1}{2}\bigl\{(a_--\Delta_-)\alpha-(a_+-\Delta_+)\beta\bigr\}
		=-a_+\beta
		-\left(\frac{\alpha}{a_-}+\frac{\beta}{a_+}\right)\varepsilon\lambda^\varepsilon
			+o(\varepsilon\lambda)
	\end{aligned}
\end{equation*}
Hence, for $\lambda<0$ and $\varepsilon\lambda\to 0$,
disregarding the exponentially small term $\theta_-^-\theta_+^+$ 
keeping only the principal term in the expansions, we infer
\begin{equation*}
	\frac{1}{2}D\approx
		-a_+ e^{a_-\alpha/\varepsilon}
		+\frac{[a]\,\varepsilon\lambda}{a_+\,a_-}\,e^{(a_-\alpha-a_+\beta)/\varepsilon}
		+a_- e^{-a_+\beta/\varepsilon}.
\end{equation*}
Therefore, $D\approx 0$ for
\begin{equation}\label{firstapprox}
	\lambda_1^{\varepsilon} \approx -
	\frac{a_+a_-}{a_+-a_-}\frac{1}{\varepsilon}
		\left(-a_+ e^{a_+\beta/\varepsilon}+a_- e^{-a_-\alpha/\varepsilon}\right)
\end{equation}
in the regime $\varepsilon\lambda$ small.}
\end{example}

Asymptotic representation \eqref{firstapprox} permits to verify the relation between the
first eigenvalue of the linearized operator and the term $\Omega^\varepsilon$,
controlling the size of $\mathcal{F}[U^\varepsilon]$ (see \eqref{defOmegaeps}).
Specifically, for the Burgers equation, \eqref{firstapprox} becomes
\begin{equation*}
	\lambda \approx 
	-\frac{1}{\varepsilon}\,u_-^2e^{-u_-\ell/\varepsilon}\,\cosh(u_-\xi/\varepsilon).
\end{equation*}
The term $\mathcal{F}[U^\varepsilon]$ given in \eqref{FUeps}
for the Burgers equation (Example \ref{example:BurgersEq}) is such that
\begin{equation*}
	\Omega^\varepsilon(\xi)\approx \frac{2}{\varepsilon}u_-^2
		\left|e^{-u_-(\ell+\xi)/\varepsilon}-e^{-u_-(\ell-\xi)/\varepsilon}\right|
	=\frac{4}{\varepsilon}u_-^2|\sinh(u_-\,\xi/\varepsilon)|\,e^{-u_- \ell/\varepsilon}.
\end{equation*}
Therefore, the estimate
\begin{equation*}
	0\leq \frac{\Omega^\varepsilon}{|\lambda^\varepsilon|}\approx 4|\tanh(u_-\,\xi/\varepsilon)|\leq 4.
\end{equation*}
holds and hypothesis \eqref{meta} is verified.
\vskip.15cm

For general scalar conservation it still possible to obtain an analogous bound.
Indeed, for $a_\pm=f'(u_\pm)$, $\alpha=\ell+\xi$ and $\beta=\ell-\xi$, expression
\eqref{firstapprox} becomes
\begin{equation*}
	\lambda_1^{\varepsilon} \approx -\left(\frac{1}{f'(u_-)}-\frac{1}{f'(u_+)}\right)^{-1}
	\frac{1}{\varepsilon}\left(-f'(u_+) e^{f'(u_+)(\ell-\xi)/\varepsilon}
		+f'(u_-) e^{-f'(u_-)(\ell+\xi)/\varepsilon}\right).
\end{equation*}
(compare with Lemma 3.2 in \cite{deGrKara98}).
The bounded for $\Omega^\varepsilon$ can be obtained by proceeding as in Section \ref{Sect:Abs},
by means of a more detailed estimate on the functions $\kappa_\pm^{\varepsilon}$ starting from 
the inequalities 
\begin{equation*}
	\begin{aligned}
		f(u)&\leq f(u_+)+f'(u_+)(u-u_+)+\frac{1}{2}c_0(u-u_+)^2	&\qquad	&u\in[u_+,u_\ast],\\
		f(u)&\leq f(u_-)+f'(u_-)(u-u_-)+\frac{1}{2}c_0(u-u_-)^2		&\qquad	&u\in[u_\ast,u_-],
	\end{aligned}
\end{equation*}
A careful (and tedious) computation of the integrals in a the corresponding approximated form 
for the implicit relation \eqref{defkappapm}, leads to the bound
\begin{equation*}
	\Omega^{\varepsilon} \leq \frac{1}{\varepsilon}\left(C_+e^{f'(u_+)(\ell-\xi)/\varepsilon}
		+C_- e^{-f'(u_-)(\ell+\xi)/\varepsilon}\right)
\end{equation*}
which, together with the asymptotic representation for $\lambda_1^\varepsilon$,
guarantees requirement \eqref{meta} in Theorem \ref{thm:metaL}.
\vskip.15cm

\subsection*{Estimate from above for the second eigenvalue.}
Controlling the location of the second (and subsequent) eigenvalue needs much more care
and, also, a number of additional assumption on the limiting behavior of the function $a^\varepsilon$
as $\varepsilon\to 0^+$.
Precisely, we suppose $a^\varepsilon\in C^0([-\ell,\ell])$ satisfies the following hypotheses:\\
\indent {\sf A1.} the function $a^\varepsilon$ is twice differentiable at any $x\neq \xi$ and
\begin{equation*}
	\frac{da^{\varepsilon}}{dx}, \frac{d^2a^{\varepsilon}}{dx^2}<0<a^{\varepsilon}
		\quad\textrm{in}\;(-\ell,\xi),
	\qquad\textrm{and}\qquad
	a^{\varepsilon},\frac{da^{\varepsilon}}{dx}<0<\frac{d^2a^{\varepsilon}}{dx^2}
		\quad\textrm{in}\;(\xi,\ell),
\end{equation*}
\indent {\sf A2.} for any $C>0$ there exists $c_{{}_{0}}>0$ such that, for any $x$
satisfying $|x-\xi|\geq c_{{}_{0}}\varepsilon$, there holds
\begin{equation*}
	|a^\varepsilon-a^0|\leq C\,\varepsilon
		\qquad\textrm{and}\qquad
	\varepsilon\left|\frac{da^\varepsilon}{dx}\right|\leq C;
\end{equation*}
\indent{\sf A3.} there exists the left/right first order derivatives of $a^\varepsilon$ at $\xi$ and 
\begin{equation*}
	\liminf_{\varepsilon\to 0^+} \,\varepsilon\left|\frac{da^\varepsilon}{dx}(\xi\pm)\right|>0
\end{equation*}
As a consequence, the function $b^\varepsilon+\varepsilon\lambda^\varepsilon$ satisfies 
a number of corresponding properties, listed in the next statement.

\begin{lemma}\label{lem:propb} 
Let the family $a^\varepsilon$ be such that hypotheses {\sf A1-2-3} are satisfied,
and let $\lambda^\varepsilon<0$ be such that
\begin{equation*}
	\inf_{\varepsilon>0}\varepsilon\lambda^\varepsilon>-\frac{1}{4}\,\alpha_{{}_{0}}^2
	\qquad\textrm{where }\; \alpha_{{}_{0}}:=\min\{|a_-|,|a_+|\}.
\end{equation*}
Then there exist $\varepsilon_{{}_{0}}>0$ such that, for $\varepsilon<\varepsilon_{{}_{0}}$,
the functions $b^\varepsilon+\varepsilon\lambda^\varepsilon$, with $b^\varepsilon$ defined in \eqref{defb},
enjoy the following properties:\\
\indent{\sf B1.} the function $b^\varepsilon+\varepsilon\lambda^\varepsilon$ is decreasing in $(-\ell,\xi)$
and increasing in $(\xi,\ell)$;\\
\indent{\sf B2.} there exist $C, c>0$ such that, for any $x$ with $|x-\xi|\geq c\,\varepsilon$
there holds $b^\varepsilon+\varepsilon\lambda^\varepsilon\geq C>0$;\\
\indent{\sf B3.} there exist the left/right limits of $b^\varepsilon+\varepsilon\lambda^\varepsilon$ at $\xi$ and 
\begin{equation*}
	\beta:=\limsup\limits_{\varepsilon\to 0^+} \bigl(b^\varepsilon(\xi\pm)+\varepsilon\lambda^\varepsilon\bigr)<0;
\end{equation*}
\end{lemma}

\begin{proof}
Property {\sf B1.} is an immediate consequence of assumption {\sf A1}, since
\begin{equation*}
	\frac{d}{dx}\left(b^{\varepsilon}+\varepsilon\lambda^\varepsilon\right)
		=\frac14\,a^{\varepsilon}\,\frac{da^{\varepsilon}}{dx}
		+\frac12\,\varepsilon\,\frac{d^2a^{\varepsilon}}{dx^2}.
\end{equation*}
From {\sf A2}, given $C>0$, for $x\leq \xi-c_{{}_{0}}\,\varepsilon$, there holds
\begin{equation*}
	\begin{aligned}
	b^{\varepsilon}+\varepsilon\lambda^\varepsilon
		&\geq \frac{1}{4}(a^\varepsilon+a^0)(a^\varepsilon-a^0)-\frac12\,\varepsilon\left|\frac{da^\varepsilon}{dx}\right|
	+\varepsilon\lambda^\varepsilon+\frac{1}{4}a_-^2\\
		&\geq \varepsilon\lambda^\varepsilon+\frac{1}{4}\alpha_{{}_{0}}^2
			-\frac{1}{2}\left(1+|a^0|\,\varepsilon+\frac12\,C\,\varepsilon^2\right)\,C
	\end{aligned}
\end{equation*}
From such inequality, by choosing $C>0$ sufficiently small, and combining with an analogous estimate 
on $(\xi+c\,\varepsilon,\ell)$, property {\sf B2.} follows.

For what concerns {\sf B3}, we observe that, since $a(\xi)=0$ and $\lambda\leq 0$, there holds 
\begin{equation*}
	\limsup\limits_{\varepsilon\to 0^+} \bigl(b^\varepsilon(\xi\pm)+\varepsilon\lambda^\varepsilon\bigr)
		\leq \limsup\limits_{\varepsilon\to 0^+} \dfrac12\,\varepsilon\,\dfrac{da^{\varepsilon}}{dx}(\xi)
		=-\liminf_{\varepsilon\to 0^+} \,\varepsilon\left|\frac{da^\varepsilon}{dx}(\xi\pm)\right|<0,
\end{equation*}
thanks to {\sf A3}.
\end{proof}

For later reference, we denote $y^\varepsilon_\pm$ the zeros of 
$b^\varepsilon+\varepsilon\lambda^\varepsilon$, 
with   $-\ell<y^\varepsilon_-<\xi<y^\varepsilon_+<\ell$. 
Since property {\sf B2} holds, we deduce that $|y^\varepsilon_\pm-\xi|\leq c_{{}_{0}}\,\varepsilon$.
\vskip.15cm

Assume the assumption of Lemma \ref{lem:propb} to hold, and let $\lambda_2^\varepsilon$ and 
$\mu_2^\varepsilon=\varepsilon\,\lambda_2^\varepsilon$ be the second eigenvalue 
of the operators ${\mathcal L}^\varepsilon_\xi$ and ${\mathcal M}^\varepsilon_\xi$, respectively, with corresponding
eigenfunctions $\phi_2^\varepsilon$ and $\psi_2^\varepsilon$.
Such eigenfunctions are linked together by the relation
\begin{equation}\label{relpsiphi}
	\psi_2^\varepsilon(x)
	=A\,\exp \left(-\frac{1}{2\varepsilon} \int_{x_\ast}^x a^{\varepsilon}(y)\,dy \right)\phi_2^\varepsilon(x)
\end{equation}
for some constants $A$ and $x_\ast$.
Since $\lambda_2^\varepsilon$ is the second eigenvalue, the functions $\phi_2^\varepsilon$ and 
$\psi_2^\varepsilon$ possess a single root located at some point $x_0^\varepsilon\in(-\ell,\ell)$.
The sign properties of $b^\varepsilon+\mu_2^\varepsilon$ described in Lemma \ref{lem:propb} 
imply that $x_0^\varepsilon\in(y_-^\varepsilon, y_+^\varepsilon)$.
Then, $\phi_2^\varepsilon$ and $\psi_2^\varepsilon$ restricted to the intervals $(-\ell,x_0^\varepsilon)$ 
and $(x_0^\varepsilon,\ell)$ are eigenfunctions relative to the first eigenvalue of the same operator 
considered in the corresponding  intervals and  with Dirichlet boundary conditions.

From now on, we drop, for shortness, the dependence on $\varepsilon$ of $\lambda_2, \phi_2, \psi_2, x_0$,
we assume, without loss of generality, $x_0\geq \xi$ and we restrict our attention  to the interval $J=(x_0,\ell)$.
Integrating on $J$, we deduce
\begin{equation*}
	 \lambda_2\,\int_{x_0}^{\ell} \phi_2\,dx=\varepsilon\,\bigl(\phi_2'(\ell)-\phi_2'(x_0)\bigr)
 	<-\varepsilon\,\phi_2'(x_0)
\end{equation*}
having chosen $\phi_2$ positive in $J$.
Assuming $\psi_2$ to be given as in \eqref{relpsiphi} with $A=1$ and $x_\ast=x_0$,
and normalized so that $\max\psi_2=1$, from the latter inequality we infer the inequality
\begin{equation}\label{estsecond1}
	|\lambda_2|>\varepsilon\,I^{-1}\,\psi_2'(x_0),
\end{equation}
where
\begin{equation*}
	I:=\int_{x_0}^{\ell} \exp \left(\frac{1}{2\varepsilon} \int_{x_0}^x a^{\varepsilon}(y)\,dy \right)\,dx
\end{equation*}
Our next aim is to deduce an estimate from above on $I_\varepsilon$ and an estimate from below for 
$\psi_2'(x_0)$, in order to get a control on the size of the second eigenvalue $\lambda_2$.

From the definition of $I_\varepsilon$, since $x_0\geq \xi$, it follows
\begin{equation*}
	\begin{aligned}
	I_\varepsilon &\leq e^{|a^\varepsilon-a^0|_{{}_{L^1}}/2\varepsilon}
		 		\int_{x_0}^{\ell} e^{a_+(x-x_0)/2\varepsilon}\,dx
	 			=\frac{2\varepsilon}{|a_+|}\,e^{|a^\varepsilon-a^0|_{{}_{L^1}}/2\varepsilon}
	 			\bigl(1-e^{a_+(\ell-x_0)/2\varepsilon}\bigr)\\
		&\leq	 \frac{2\varepsilon}{|a_+|}\,e^{|a^\varepsilon-a^0|_{{}_{L^1}}/2\varepsilon}
			\leq C\,\varepsilon
	\end{aligned}
\end{equation*}
whenever $|a^\varepsilon-a^0|_{{}_{L^1}}\leq C\,\varepsilon$.
Thus, estimate \eqref{estsecond1} provisionally becomes 
\begin{equation}\label{estsecond2}
	|\lambda_2|>C\,\psi_2'(x_0)
\end{equation}
for some positive constant $C$, independent on $\varepsilon$.

Let the value $x_M$ be such that $\psi_2(x_M)=1$, minimum with such property.
From the assumption on the function $b^\varepsilon+\varepsilon\,\lambda$, it follows
$x_M\in(x_0,y_+)$.
Then there exists $x_L\in(x_0,x_M)$ such that
\begin{equation*}
	\psi_2'(x_L)=\frac{1}{x_M-x_0}\geq \frac{1}{y_+-\xi}\geq \frac{1}{c_{{}_{0}}\varepsilon}.
\end{equation*}
Since the function $\psi$ is concave in the interval $(x_0,y_+)$, we deduce 
\begin{equation*}
	\psi_2'(x_0)\geq \psi_2'(x_L) \geq \frac{1}{c_{{}_{0}}\varepsilon}.
\end{equation*}
Plugging into \eqref{estsecond2}, we end up with
$|\lambda_2|\geq C/\varepsilon$, for some $C$ independent on $\varepsilon$.

As a consequence, we can state a result relative to the second eigenvalue $\lambda_2$.
\vskip.15cm

\begin{proposition}\label{prop:secondevalue}
Let $a^\varepsilon$ be a family of functions sastisfying {\sf A1-2-3} 
then there exists $C>0$ such that  $\lambda^\varepsilon_2\leq -C/\varepsilon$
for any $\varepsilon$ sufficiently small.
\end{proposition}

\subsection*{Spectral estimates}
Collecting the results of Propositions \ref{prop:firstevalue} and \ref{prop:secondevalue} give a complete
description for the spectrum of operator $\mathcal{L}^\varepsilon$ for small $\varepsilon$, under
assumptions {\sf A0-1-2-3} on the family $a^\varepsilon$.
\vskip.15cm

\begin{corollary}\label{prop:spectral}
Let $a^\varepsilon$ be a family of functions satisfying the assumptions {\sf A0-1-2-3} for some 
$\xi\in(-\ell,\ell)$, $a_+<0<a_-$.
Then there exist $C>0$ such that 
\begin{equation*}
	\lambda^\varepsilon_k\leq -C/\varepsilon
		\qquad\textrm{and}\qquad
	-C e^{-C/\varepsilon}\leq \lambda^\varepsilon_1<0.
\end{equation*}
for any $k\geq 2$.
\end{corollary}

Hypotheses {\sf A0-1-2-3} are satisfied in the case of a family of function $a^\varepsilon$
that is a (small) perturbation of a function $\bar a^\varepsilon$ with the form
\begin{equation*}
	\bar a^\varepsilon(x)=A_-\left(\frac{x-\xi}{\varepsilon}\right)\chi_{{}_{(-L,\xi)}}(x)
		+A_+\left(\frac{x-\xi}{\varepsilon}\right)\chi_{{}_{(\xi,L)}}(x).
\end{equation*}
for some decreasing smooth bounded functions $A_\pm$, bounded together with their first and 
second order derivatives, and such that $A_\pm(\pm\infty)=a_\pm$ and $A_\pm'(\pm\infty)=0$.

\section*{Acknowledgment}
The authors are thankful to the anonymous referees for their valuable comments
and suggestions which helped them improving the content and the overall organization
of the article.

\end{document}